\documentclass[12pt]{amsart}
\usepackage{color}
\usepackage{graphicx}
\usepackage{subfigure}
\usepackage{amssymb}
\usepackage{amsmath}
\usepackage{multirow}
\usepackage{enumerate}
\usepackage{epstopdf}
\usepackage{cite,stmaryrd,txfonts}
\usepackage{tikz}

\input{undertilde}
\usetikzlibrary{snakes}

\usepackage{verbatim}

\usepackage{epic}%% package for dash line \dashline

\usepackage{empheq}

\usepackage{ulem}

\newcommand{\bs}{\boldsymbol}

\newtheorem{algorithm}{Algorithm}

\def\cequiv{\raisebox{-1.5mm}{$\;\stackrel{\raisebox{-3.9mm}{=}}{{\sim}}\;$}}
\def\dx{\mathrm{dx}}

\def\uphi{\undertilde{\varphi}}
\def\upsi{\undertilde{\psi}}

\def\rot{{\rm rot}}

\def\vgm12{\bs{V}^{1+,2}_{\gamma,M}}

\newtheorem{theorem}{Theorem}
\newtheorem{remark}[theorem]{Remark}

\newtheorem{lemma}[theorem]{Lemma}
\newtheorem{corollary}[theorem]{Corollary}
\newtheorem{definition}[theorem]{Definition}

\newcounter{mnote}
\let\oldmarginpar\marginpar
\renewcommand\marginpar[1]{\-\oldmarginpar[\raggedleft\footnotesize #1]%
  {\raggedright\footnotesize #1}}

\begin{document}
%\xiaoerhao

\title[A multi-level mixed method for the biharmonic eigenvalue problem]{A multi-level mixed element method for the eigenvalue problem of biharmonic equation}
\author{Shuo Zhang}
\address{LSEC, Institute of Computational Mathematics and Scientific/Engineering Computing, Academy of Mathematics and System Sciences, Chinese Academy of Sciences, Beijing 100190, People's Republic of China}
\email{szhang@lsec.cc.ac.cn}
\thanks{S. Zhang is partially supported by the National Natural Science Foundation of China with Grant No. 11471026 and National Centre for Mathematics and Interdisciplinary Sciences, Chinese Academy of Sciences.}

\author{Yingxia Xi}
\address{LSEC, Institute of Computational Mathematics and Scientific/Engineering Computing, Academy of Mathematics and System Sciences, Chinese Academy of Sciences, Beijing 100190, People's Republic of China}
\email{yxiaxi@lsec.cc.ac.cn}

\author{Xia Ji}
\address{LSEC, Institute of Computational Mathematics and Scientific/Engineering Computing, Academy of Mathematics and System Sciences, Chinese Academy of Sciences, Beijing 100190, People's Republic of China}
\email{jixia@lsec.cc.ac.cn}
\thanks{X. Ji is supported by the National Natural Science Foundation of China (No. 11271018, No. 91230203) and the Special Funds for National Basic Research Program of China (973 Program 2012CB025904), and National Centre for Mathematics and Interdisciplinary Sciences, Chinese Academy of Sciences.}



\begin{abstract}

In this paper, we discuss approximating the eigenvalue problem of biharmonic equation. We first present an equivalent mixed formulation which admits amiable nested discretization. Then, we construct multi-level finite element schemes by implementing the algorithm as in \cite{LinXie2015} to the nested discretizations on series of nested grids. The multi-level mixed scheme for biharmonic eigenvalue problem possesses optimal convergence rate and optimal computational cost. Both theoretical analysis and numerical verifications are presented. 
\end{abstract}

%\date{November}
\maketitle

%\tableofcontents

%\newpage

%
%\begin{enumerate}
%\item an equivalent mixed formulation
%\item user-friendly discretisation
%\item multi-level schemes
%\item numerical examples show guaranteed upper and lower bounds
%\end{enumerate}
%
%motivations
%\begin{enumerate}
%\item multi-level schemes, for the optimal accuracy
%\item mixed formulation: 1. lower-degree, 2. nested spaces
%\item generalized symmetric operator
%\item upper and lower bound
%\end{enumerate}
%
%\begin{enumerate}
%\item reduce the stiffness
%\item nested discretisation of the eigenvalue problem
%\begin{enumerate}
%\item reduce the order of differential, nestedness of the finite element spaces
%\item well-defined discretisation
%\item nestedness of the topology
%\end{enumerate}
%
%\end{enumerate}

%
%
%
\section{Introduction}

%
%\subsection{the purpose of the paper; the origin and application of fourth order eigenvalue problem}

The eigenvalue problem of the biharmonic equation (biharmonic eigenvalue problem) is one of the fundamental model problems in linear elasticity, and can find applications in, e.g., modelling the vibration of thin plates. There has been a long history on developing the finite element methods of the biharmonic eigenvalue problem, and many schemes have been proposed for discretization \cite{BrennerMonkSun2015,ChenLin2007,JiaXieYinGao2007,DavidRodolfo2009}, computation of guaranteed upper and lower bounds\cite{Carsten2014,Yang.Y;Lin.Q;Bi.H;Lin.Q2012,Hu.J;Huang.Y;Lin.Q2015,Hu.J;Huang.Y;Shen.Q2015}, and adaptive method and its convergence analysis \cite{Gallistl.D2015}. This paper is devoted to studying the multi-level efficient method of the biharmonic eigenvalue problem. Specifically, we present a discretization scheme which preserves the nested essence on nested grids, and then construct a multi-level algorithm based on the scheme. The cost of the multi-level algorithm versus the intrinsic accuracy of the scheme is asymptotically optimal.

As well known, the multi-level algorithm based on nested essence has been a key tool in computational mathematics and scientific computing fields. For the eigenvalue problem, many multi-level algorithms have been designed and implemented. For example, there are several successful methods for the Poisson eigenvalue problem. The two-grid method has been proposed and analyzed by Xu-Zhou in \cite{XuZhou2001}. The idea of the two-grid method is related to the ideas in [23, 24] for nonsymmetric or indefinite problems and nonlinear elliptic equations.
Since then, many numerical methods for solving eigenvalue problems based on the idea of the two-grid method are developed (see, e.g.,\cite{BiYang2011,ChienJeng2006,DaiZhou2007,Kolman2015,LinYang2011,YangJiangZhangWangBi2012}). A type of multi-level correction scheme is presented by Lin-Xie \cite{LinXie2015} and Xie \cite{Xie.HJCP}. The method is a type of operator iterative method (see, e.g, \cite{Lin.Q1979,XuZhou2001,Zhou.A2010}).  Besides, Xie \cite{Xie2015} presents a multi-level correction scheme, and the guaranteed lower bounds of the eigenvalues can be obtained. The correction method for eigenvalue problems in these papers are based on a series of finite element spaces with different approximation properties related to the multi-level method (cf. \cite{Xu.J1992SR}). With the proposed methods, the eigenvalue problem is transformed to an eigenvalue problem on the coarsest grid and a series of source problem on the fine grids. The scheme can be proved asymptotically optimal. The same strategy can be implemented on the Stokes equation, and similar asymptotic optimality is constructed \cite{LinLuoXie2013}. These works mentioned above have indeed presented a framework of designing multi-level schemes which works well for the elliptic eigenvalue problem and stable saddle point problem, provided a series of subproblems with intrinsic nestedness.

In contrast to the second order problem, the multi-level method for the biharmonic eigenvalue problem has seldom been discussed, due to the lack of nested subproblems. Indeed, when we consider the primal formulation of the biharmonic problem, the high stiffness of the Sobolev space $H^2$ makes it difficult to construct nested discretizations. Besides spline-type elements, the rectangular BFS element \cite{Bogner1965} is the only element which can form nested finite element spaces on nested grids; a multi-level algorithm has been designed based on BFS element for fourth order problems on rectangular grids \cite{Ji2014}.  Moreover, elements that are able to form nested spaces are proved to be conforming ones; therefore, people can not obtain guaranteed lower bounds of eigenvalues with these elements. One way for this situation is to loose the stiffness of the finite element spaces. Mixed element method is then frequently used, and several schemes for the biharmonic eigenvalue problem with polynomials of low degree have been designed\cite{AndreevLazarov2005,Ishihara1978}. Also, some discretization schemes of mixed type for boundary value problems can be naturally utilized for the eigenvalue problem; we refer readers to \cite{Boffi.D;Brezzi.F;Fortin.M2013}  for related discussion.  However, we have to remark that the order-reduced nestedness discretizations is still not straightforward. For example, the Ciarlet-Raviart formulation\cite{Ciarlet.P;Raviart.P1974} admits us to discretize the biharmonic operator with piecewise continuous linear polynomials. However, as this formulation is stable on the space pair $H^1_0(\Omega)\times H^{-1}(\Delta,\Omega)$ \cite{Bernardi;Girault;Maddy1992}, the inheritance of the {\it topology} onto the finite element space is an issue, and the finite element spaces on nested grids are not {\it topologically} nested. The same problem is encountered for some other mixed formulations which introduce direct auxiliary variables, such as \cite{Falk.R1978,Johnson.C1973,Hermann.L1967,Hellan.K1967,Krendl.W;Rafetseder.K;Zulehner.W}. More discussion can be found in  \cite{Li.Z;Zhang.S2016}. These may explain why few multi-level scheme is discussed for the biharmonic eigenvalue problem.

In this paper, we seek to implement multi-level strategy by constructing amiable nested finite element discretization for the biharmonic eigenvalue problem. We first introduce a mixed formulation whose corresponding source problem is discussed in \cite{Li.Z;Zhang.S2016} and \cite{Girault;Raviart1986}. This mixed formulation is stable on Sobolev spaces of zero and first orders (cf. Lemma \ref{eq:contiiso} below). As the stiffness is loosened, polynomials of low degree are enough for its discretization, and optimal accuracy can be expected. Therefore, it admits discretizations that are nested algebraically and {\it topologically}. Secondly, we construct a family of multi-level schemes for the mixed formulation of the eigenvalue problem. The multi-level algorithms for biharmonic eigenvalue problem possess optimal accuracy and optimal computational cost.

For the proposed algorithms, both theoretical analysis and numerical verification are given. We remark that, though the multi-level strategy is essentially the same as the one used by Lin-Xie \cite{LinXie2015,LinLuoXie2013,Xie.HJCP,Ji2014}, the theoretical analysis is not directly by the same virtue. Actually, if we separate the ``primal variables" from ``Lagrangian multipliers", we will find the skeleton bilinear form is not coercive on the primal variables nor on the Lagrangian multipliers. This makes the classical theory of the spectral approximation of the saddle-point problems (cf. \cite{LinLuoXie2013,Mercier.B;Osborn.J;Rappaz.J;Raviart.P1981,Boffi.D;Brezzi.F;Gastaldi.L1997}) not directly usable in the present paper. A precise discussion can be found in Remark \ref{rem:vsstokes}. Meanwhile, because of the saddle-point-type essence, the problem is also different from the Steklov eigenvalue problem discussed in \cite{Xie.HIMA}. We therefore construct different theory framework and interpret the eigenvalue problem in mixed formulation as the eigenvalue problem of a generalized symmetric operator rather than a self-adjoint one, and accomplish the theoretical analysis. The differences between our theory and the existing theory for elliptic or saddle point problems include: (1) we represent some existing results which are originally in variational formulation  into operator formulation, and then present error estimation in that context; the operator formulation can bridge the gap between the biharmonic problem and the classical theory of spectral approximation, and can avoid complicated appearance especially for the mixed formulation; (2) we figure out some properties of generalized symmetric operators which are not necessarily self-adjoint; and (3) in our theory, we do not try to interpret the problem as a restrained problem on primal variables or one on Lagrangian multipliers, which is usually done for saddle-point problem; this makes the algorithm construction and theoretical analysis more straightforward. 

The remaining of the paper is organized as follows. In Section \ref{sec:sa}, we present the theory of spectral approximation of the generalized symmetric operators. Some existing results are restated and re-proved, and some new results are presented. In Section \ref{sec:mm}, we present a mixed formulation of the biharmonic eigenvalue problem, and construct its (single-level) discretization schemes. A multi-level algorithm is then constructed accordingly. Both the single- and multi-level algorithms are optimal in accuracy, and the multi-level one also possesses optimal computational cost. The theoretical proof is obtained under the framework discussed in Section \ref{sec:mm}. Numerical examples are then given in Section \ref{sec:ne} with respect to both single- and multi-level methods. Finally, in Section \ref{sec:cr}, some concluding remarks and further discussion are given.

%\newpage
%
%
%
\section{Spectral approximation of generalized symmetric compact operators}
\label{sec:sa} % sa = Spectral Approximation

In this section, we present some known and new results, including
\begin{itemize}
\item[--] an estimate of spectral projection operator (Lemma \ref{lem:proj});
\item[--] an multi-level algorithm (Algorithm \ref{alg:mlalg}) and its convergence estimate (Theorem \ref{thm:cmla});
\item[--] spectral approximation of generalized symmetric operator (Lemmas \ref{lem:listasT},\ref{lem:listasTdis} and \ref{lem:estgsev});
\item[--] corresponding results in variational form (Lemma \ref{lem:listTevvf}, Algorithm \ref{alg:variational} and Theorem \ref{thm:cmlavf}).
\end{itemize}
Some bibliographic comments are given around.
\subsection{Preliminaries}
In this subsection, we collect some preliminaries from Chapter II of \cite{Babuska.I;Osborn.J1991}.
%
%Let $A:X\to X$ be a compact operator on a complex Banach space $X$ with norm $\|\cdot\|_X=\|\cdot\|$. We denote by $\rho(A)$ the resolvent set of $A$, i.e., the set
%$$
%\rho(A)=\{z:z\in\mathbb{C}, (zI-A)^{-1}\ \mbox{exists\ as\ a\ bounded\ operator\ on}\ X\},
%$$
%and by $\sigma(A)$ the spectrum of $A$, i.e., the set $\sigma(A)=\mathbb{C}\setminus \rho(A)$. For any $z\in\rho(A)$, $R_z(A)=(zI-A)^{-1}$ is the resolvent operator. $\sigma(A)$ is countable with no nonzero limit points; nonzero numbers in $\sigma(A)$ are eigenvalues; and if zero is in $\sigma(A)$, it may or may not be an eigenvalue.
%
%Let $\mu\in\sigma(A)$ be nonzero. There is a smallest integer $\alpha$, called the ascent of $\mu-A$, such that $N((\mu I-A)^\alpha)=N((\mu I-A)^{\alpha+1})$, where $N$ denotes the null space.

Let $H$ be a Hilbert space, and $T$ be a compact operator on $H$. Let $\mu$ be a nonzero eigenvalue of $T$ with algebraic multiplicities $m$. Denote the eigenspace $M(\mu):=\{u\in H:Tu=\mu u\}$. Let $\Gamma_\mu$ be a circle on the complex plane centered at $\mu$ which encloses no other points of $\sigma(T)$. Let $\{T_h\}_{0<h\leqslant 1}$ be a family of compact operators that converges to $T$ in norm. Then for $h$ sufficiently small, there exist $m$ eigenvalues of $T_h$, counting multiplicities, located inside $\Gamma_\mu$. Denote them by $\mu_{i,h}, i=1,\cdots,m$. Let $u_{i,h}$ be the eigenvectors of $T_h$ with respect to $\mu_{i,h}$. Denote $M_h(\mu):={\rm{ span}}\{u_{i,h}\}_{i=1,\dots,m}$. Then $M_h(\mu)$ is the approximation of $M(\mu)$, measured by the {\bf gap} between them.

%The lemma below can indeed be found in
%\begin{lemma}\label{lem:eigappr}(\cite{Babuska.I;Osborn.J1991}, P 684.) Given $\epsilon>0$, there is an $h(\epsilon)$, such that for $h<h(\epsilon)$, there are $m$ eigenvalues of $T_h$: $\mu_1(h),\mu_2(h),\dots,\mu_m(h)$, counting according to algebraic multiplies, such that $|\mu_j(h)-\mu|<\epsilon$, $j=1,\dots,m$.
%\end{lemma}
A {\bf gap} between two closed subspaces $M$ and $N$ of a Banach space $X$ is defined by
$$
\hat{\delta}(M,N)=\max(\delta(M,N),\delta(N,M)),  \mbox{with}\ \delta(M,N)=\sup_{x\in M, \|x\|=1}{\rm dist}(x,N).
$$

\begin{lemma}\label{lem:delta}(\cite{Babuska.I;Osborn.J1991,Kato1958} )
If $\dim M=\dim N<\infty$, then $\delta(N,M)<\delta(M,N)[1-\delta(M,N)]^{-1}$.
\end{lemma}

\begin{lemma}\label{lem:hbthm7.1} % HandBook THeoreM 7.1
(\cite{Babuska.I;Osborn.J1991}, Theorem 7.1.) There is a constant $C$ independent of $h$, such that
$$
\hat{\delta}(M(\mu),M_h(\mu))\leqslant C\|(T-T_h)|_{M(\mu)}\|,
$$
for small $h$, where $(T-T_h)|_{M(\mu)}$ denotes the restriction of $T-T_h$ to $M(\mu)$.
\end{lemma}
Define the projection operators with respect to $\mu$ by
\begin{equation}
E=\frac{1}{2\pi i}\int_{\Gamma_\mu} R_z(T) dz=\frac{1}{2\pi i}\int_{\Gamma_\mu} (z-T)^{-1} dz,\quad E_h=\frac{1}{2\pi i}\int_{\Gamma_\mu} R_z(T_h) dz=\frac{1}{2\pi i}\int_{\Gamma_\mu} (z-T_h)^{-1} dz.
\end{equation}
Then ${\rm range}(E)=M(\mu)$, and ${\rm range}(E_h)=M_h(\mu)$. We refer to \cite{Babuska.I;Osborn.J1991} for more discussion.

\subsection{Spectral approximation by the aid of projection operator}

For $G$ a subspace of $H$ with $H=G\oplus G^c$, denote by $P_G$ the projection operator onto $G$ along $G^c$. Let $\{G_h\}$ be a sequence of subspaces of $H$, with the indices  $h\to 0$, and $G_h\to H$.  Define $T_{G_h}:=P_{G_h}T$, then $\{T_{G_h}\}$ are approximations of $T$. We know that if $\|P_{G_h}u-u\|_H\to 0$ as $h\to 0$ for any $u\in H$,  then $\|T-T_{G_h}\|_H\to 0$ as $h\to 0$.

We write for short $P_h$ the projection onto $G_h$, and $T_h:=P_hT$. We assume $T_h\to T$ in norm as $h\to 0$. Corresponding to $M(\mu)$ and $M_h(\mu)$, we have the lemma below. 

\begin{lemma}\label{lem:proj}
There is a constant $C_\mu$, such that
\begin{eqnarray}
&\|u-E_hu\|_H\leqslant C_\mu\|(I-P_h)u\|_H,\ \forall\,u\in M(\mu),\label{eq:ebddbp}
\\
 &\|T(u-E_hu)\|_H\leqslant C_\mu\|T(I-P_h)u\|_H,\ \forall\,u\in M(\mu).\label{eq:tebddbtp}
\end{eqnarray}
\end{lemma}

\begin{proof}
Direct calculation leads to that
\begin{eqnarray}
&u-E_hu=Eu-E_hu=\frac{1}{2\pi i}\int_{\Gamma_\mu}(z-T_h)^{-1}(T-T_h)\frac{u}{z-\mu}dz\nonumber
\\
&=\frac{\mu}{2\pi i}\int_{\Gamma_\mu}(z-T_h)^{-1}(I-P_h)\frac{u}{z-\mu}dz=\frac{\mu}{2\pi i}\int_{\Gamma_\mu}(z-T_h)^{-1}(I-P_h)^2\frac{u}{z-\mu}dz,
\end{eqnarray}
where it has been used that $(T-T_h)u=\mu(I-P_h)u$, and $(I-P_h)^2=(I-P_h)$. Thus
$$
\|u-E_h(\mu)u\|_H\leqslant \frac{1}{2\pi}[2\pi{\rm rad}(\Gamma_\mu)]\sup_{z\in\Gamma_\mu,h>0}\|(z-T_h)^{-1}\|_H\frac{\|(T-T_h)u\|_H}{{\rm rad}(\Gamma_\mu)}=|\mu|\sup_{z\in\Gamma_\mu,h>0}\|(z-T_h)^{-1}\|_H\|(I-P_h)u\|_H.
$$
Now $\|T-T_h\|_H\to 0$ implies that
$$
|\mu|\sup_{z\in\Gamma_\mu,h>0}\|(z-T_h)^{-1}\|_H<\infty.
$$
This proves \eqref{eq:ebddbp}. Further, note that
$$
(z-T_h)^{-1}=(I-(z-T_h)^{-1}(I-P_h)T)(z-T)^{-1},
$$
and we have
\begin{eqnarray}
&T(u-E_hu)=\frac{\mu}{2\pi i}\int_{\Gamma_\mu} T(z-T_h)^{-1}(I-P_h)\frac{u}{z-\mu}dz
\nonumber\\
& =\frac{\mu}{2\pi i}\int_{\Gamma_\mu} (I-T(z-T_h)^{-1}(I-P_h))T(z-T)^{-1}(I-P_h)\frac{u}{z-\mu}dz
\nonumber\\
&=\frac{\mu}{2\pi i}\int_{\Gamma_\mu} (I-T(z-T_h)^{-1}(I-P_h))(z-T)^{-1}T(I-P_h)\frac{u}{z-\mu}dz.\qquad\ \
\end{eqnarray}
Thus
$$
\|T(u-E_hu)\|_H\leqslant \frac{|\mu|}{2\pi}[2\pi{\rm rad}(\Gamma)]\sup_{z\in{\Gamma_\mu},h>0}\|(I-T(z-T_h)^{-1}(I-P_h))(z-T)^{-1}\|_H\frac{\|T(I-P_h)u\|_H}{{\rm rad}(\Gamma)}.
$$
Since $\|(z-T_h)^{-1}\|_H$ and $\|(z-T)^{-1}\|_H$ are uniformly bounded for $z\in\Gamma$ and $h>0$, we obtain
$$
\|T(u-E_hu)\|_H\leqslant C_\mu \|T(I-P_h)u\|_H.
$$
The proof is completed.
\end{proof}
\begin{remark}
Inequality \eqref{eq:ebddbp} is $($3.16a$)$ of \cite{Babuska1989}, while \eqref{eq:tebddbtp} is a generalisation of $($3.16c$)$ of \cite{Babuska1989}.
\end{remark}

%
%
%The difference between $M(\mu)$ and $M_h(\mu)$ is measured by ``gap".
%
%
%
%
%The spectral projection associated with $T$ and $\mu$ is
%$$
%E=E(\mu)=\frac{1}{2\pi i}\int_\Gamma R_z(T) dz,
%$$
%where $R_z(T):=(z-T)^{-1}$ is the resolvent operator of $T$.
%
%Let $\{T_h\}_{0<h\leqslant 1}$ be a family of compact operators, such that $T_h\to T$ in norm as $h\searrow 0$. Then for $h$ sufficiently small, $\Gamma\subset \rho(T_h)$ and the spectral projection
%$$
%E_h=E_h(\mu)=\frac{1}{2\pi i}\int_\Gamma R_z(T_h) dz
%$$
%exists, $E_h\to E$ in norm, and  $\dim(Range(E_h(\mu)))=\dim(Range(E(\mu)))=m$. $E_h$ is the spectral projection associated with $T_h$ and the eigenvalues of $T_h$ which lie in $\Gamma$ and is a projection onto the direct sum of the spaces of generalised eigenvectors corresponding to these eigenvalues.
%
%$R(E)$ and $R(E_h)$ are invariant subspaces for $T$ and $T_h$, respectively, and $TE=ET$ and $T_hE_h=E_hT_h$. $\{R_z(T_h):z\in\Gamma, h\ \mbox{small}\}$ is bounded.
%

%\begin{proof}
%\begin{multline}
%\|T(I-P_h)u\|_H\cequiv\sup_{v\in H}\frac{|a(T(I-P_h)u,v)|}{\|v\|_H}=\sup_{v\in H}\frac{|a(T(I-P_h)^2u,v)|}{\|v\|_H}
%\\
%=\sup_{v\in H}\frac{|a((I-P_h)u,(I-P_h)Tv)|}{\|v\|_H}\leqslant C\|T-T_h\|_H\|(I-P_h)u\|_H.
%\end{multline}
%\end{proof}

%\begin{remark}
%Evidently, $\|T(u-E_h(\mu)u)\|_H\leqslant C\|T-T_h\|_H\|u-E_h(\mu)u\|_H$, provided {\bf H1} and {\bf H2}.
%\end{remark}
%

%
%
\subsection{A multi-level algorithm for eigenvalue problem with projection approximation}

The algorithm is the same as the algorithms employed in \cite{Xie.HJCP,LinXie2015,LinLuoXie2013}, but is rewritten with respect to a general context of operator. The error estimation is then reformed accordingly.

\begin{algorithm}\label{alg:mlalg}
A multi-level algorithm for $k$ eigenvalues of $T$.

%\fbox{
%\begin{minipage}{1.\textwidth}
%
\begin{description}
\item[Step 0] Construct a series of nested spaces $G_0\subset G_1\subset\dots\subset G_N\subset H$. Set $\widetilde{G}_0=G_0$.
\item[Step 1] For $i=1:1: N$, generate auxiliary spaces $\widetilde{G}_i$ recursively.
\begin{description}
\item[Step 1.i.1]
Define projection operators $\widetilde{P}_{i-1}:H\to \widetilde{G}_{i-1}$, and solve eigenvalue problem for its first $k$ eigenpairs $\{(\tilde{\mu}_j^{i-1},\tilde{u}_j^{i-1})\}_{j=1,\dots,k}$
$$
 \widetilde{P}_{i-1}T\tilde{u}=\tilde{\mu}\tilde{ u};
$$
\item[Step 1.i.2] Define projection operators $P_i:H\to G_i$. Compute
$$
\hat{u}_j^i=\frac{1}{\tilde{\mu}_j^{i-1}}P_iT\tilde{u}_j^{i-1},\ \ j=1,\dots,k;
$$
\item[Step 1.i.3] Set
$$
\widetilde{G}_i=G_0+{\rm span}\{\hat{u}_j^i\}_{j=1}^k.
$$
\end{description}
\item[Step 2] Define projection operators $\widetilde{P}_N:H\to \widetilde{G}_N$, solve eigenvalue problem for its first $k$ eigenpairs $\{(\tilde{\mu}_j^N,\tilde{u}_j^N)\}_{j=1,\dots,k}$:
$$
\widetilde{P}_NT\tilde{u}=\tilde{\mu}\tilde{ u}.
$$

\end{description}

%\end{minipage}
%}
\end{algorithm}

\begin{remark}
In the algorithm, the ``first" $k$ eigenvalues imply the $k$ modulus-biggest eigenvalues. The main work of the algorithm is to solve eigenvalue problems of $\widetilde{T}_i:=\widetilde{P}_iT$ and to compute the action of $T_i:=P_iT$ on every level.
\end{remark}

Let $\mu$ be a nonzero eigenvalue of $T$ with multiplicity $m$, and denote $M(\mu)=\{u\in H:Tu=\mu u,\ \|u\|_H=1\}.$ Let $\tilde{\mu}_j^i$ and $\tilde{u}_j^i$, $j=1:m$, $i=0:N$ be the eigenpairs generated by the algorithm as approximations to $\mu$ and $M(\mu)$. Specifically, denote $\widetilde{M}_i(\mu):={\rm span}\{\tilde{u}_j^i\}_{j=1}^m.$

\paragraph{\bf Stability Constant} Let $\{\varphi_i\}_{i=1}^n\subset H$ be $n$ unit vectors. Denote the stability constant of $\{\varphi_i\}_{i=1}^n$ by
\begin{equation}
%\displaystyle
\theta(\varphi_1,\dots,\varphi_n):=\displaystyle\inf_{\alpha\in\mathbb{R}^n,\alpha\neq \bf{0}}\frac{\|\sum_{i=1}^n\alpha_i\varphi_i\|_H^2}{\sum_{i=1}^n\|\alpha_i\varphi_i\|_H^2}.
\end{equation}
The stability constant of $\{\varphi_i\}_{i=1}^n$ denotes to what extent the vectors are nearly orthogonal. If $\theta(\varphi_1,\dots,\varphi_n)=1$, then $\{\varphi_i\}_{i=1}^n$ are orthogonal to each other, and if $\theta(\varphi_1,\dots,\varphi_n)=0$, then $\{\varphi_i\}_{i=1}^n$ are linearly dependent.

\begin{lemma}\label{lem:scdist}
Let $\varphi\in {\rm span}\{\varphi_i\}_{i=1}^n$ be a unit vector, and $V\neq \Phi$ be a closed subspace of $H$. Then
\begin{equation}
{\rm dist}(\varphi,V)\leqslant \sqrt{2n\,\theta(\varphi_1,\dots,\varphi_n)^{-1}}\max_{1\leqslant i\leqslant n}{\rm dist}(\varphi_i,V).
\end{equation}
Here we define $\theta(\varphi_1,\dots,\varphi_n)^{-1}=\infty$, if $\theta(\varphi_1,\dots,\varphi_n)=0$.
\end{lemma}
\begin{proof}
Let $v_i\in V$ such that $\|\varphi_i-v_i\|_H={\rm dist}(\varphi_i,V)$. Let $\varphi=\sum_i\beta_i\varphi_i$, such that $\|\varphi\|_H=1$. Then
$$
{\rm dist}(\varphi,V)^2\leqslant \|\varphi-\sum_i\beta_iv_i\|_H^2=\|\sum_i\beta_i(\varphi_i-v_i)\|_H^2\leqslant \sum_{i,j}|\beta_i\beta_j|\|\varphi_i-v_i\|_H\|\varphi_j-v_j\|_H\leqslant [2n\sum_i\beta_i^2]\max_i\|\varphi_i-v_i\|_H^2.
$$
The proof is completed by the definition of $\theta(\varphi_1,\dots,\varphi_n)$.
\end{proof}

%\paragraph{\bf Assumption AAOA} There is a constant $C$, uniform for $\{G_i\}$, such that $\displaystyle\|u-P_Gu\|_H\leqslant C\inf_{v\in G}\|u-v\|_H$.

\begin{theorem} \label{thm:cmla}% Convergence of multi-level Algorithm
Assume $G_0$ is big enough, such that $\delta(H,G_0)$ is sufficiently small. Assume for the projections that $\displaystyle\min_{1\leqslant l\leqslant N}\inf_{u\in H\setminus G_l,v\in G_l}\frac{\|u-v\|_H}{\|u-P_lu\|_H}\geqslant C_0$, and assume for the computed eigenvectors that $\displaystyle\inf_{1\leqslant l\leqslant N}\theta(\tilde{u}_1^l,\dots,\tilde{u}_m^l)\geqslant \theta_0$. There exist constants $\beta_1$ and $\beta_2$ dependent of $\mu$, $C_0$ and $\theta_0$, such that,
\begin{equation}
\delta(M(\mu),\widetilde{G}_N)\leqslant \beta_1\sum_{l=0}^{N}[\prod_{j=l}^{N-1}(\beta_2\|T-T\widetilde{P}_{j}\|_H)]\delta(M(\mu),G_l),
\end{equation}
%\begin{equation}
%\delta(M(\mu),\tilde{G}_N)\leqslant C_{\mu,3}\sum_{l=0}^{N}C_{\mu,4}^{(N-l)}[\prod_{j=l}^{N-1}\|T-\tilde{T}_{j}\|_H]\delta(M(\mu),G_l).
%\end{equation}
%\begin{equation}\label{thm:conves}
%\hat{\delta}(M(\mu),\widetilde{M}_N(\mu))\leqslant  \beta_2\sum_{l=0}^N \left[(\beta_1 \|T-TP_0\|_H)^{N-l}\delta(M(\mu),G_l)\right],
%\end{equation}
%and that
%\begin{equation}\label{thm:convev}
%|\mu-\tilde\mu_j^N|\leqslant \beta_2\left(\sum_{l=0}^N \left[(\beta_1 \|T-TP_0\|_H)^{N-l}\delta(M(\mu),G_l)\right]\right)^2.
%\end{equation}
%where
%{\color{red}
%$$
%\delta_{P_l}(M(\mu),G_l):=\sup_{u\in M(\mu)}\|u-P_lu\|_H,\ \ l=0,1,\dots,N.
%$$
%}
%
\end{theorem}

\begin{proof}
By lemma \ref{lem:hbthm7.1},
$$
\hat{\delta}(M(\mu),\widetilde{M}_0(\mu))\leqslant C\sup_{u\in M(\mu)}\|(T-T_0)u\|_H\leqslant C_{\mu,1}\delta(M(\mu),\widetilde{G}_0).
$$
Given $\{\tilde{u}_j^0\}_{j=1}^m$, there exists $\{u_j^0\}\subset M(\mu)$, such that $\gamma_j^0\tilde{u}_j^0=E_0u_j^0$, where $|\gamma_j^0-1|\leqslant C_1\|u_j^0-P_0u_j^0\|_H$ is guaranteed arbitrarily small. Set $\alpha_j^0=\gamma_j^0\tilde{\mu}_j^0/\mu$, then
\begin{eqnarray}
&\|\alpha_j^0\hat{u}_j^1-P_1u_j^0\|_H=\|(\alpha_j^0/\tilde{\mu}_j^0)T_1\tilde{u}_j^0-(1/\mu)P_1Tu_j^0\|_H
\nonumber\\
&=\left|1/\mu\right|\|P_1T(\gamma_j^0\tilde{u}_j^0-u_j^0)\|_H
\leqslant |1/\mu|\|T(u_j^0-E_0u_j^0)\|_H\leqslant C_{\mu,2}\|T-T\widetilde{P}_0\|_H\delta(M(\mu),\widetilde{G}_0).
\end{eqnarray}
Therefore,
$$
\|\alpha_j^0\hat{u}_j^1-u_j^0\|_{H}\leqslant \|\alpha_j^0\hat{u}_j^1-P_1u_j^0\|_{H}+\|P_1u_j^0-u_j^0\|_{H}\leqslant C_{\mu,2}\|T-T\widetilde{P}_0\|_H\delta(M(\mu),\widetilde{G}_0) + \delta(M(\mu),G_1).
$$
{Since $\theta(\tilde{u}_1^0,\dots,\tilde{u}_m^0)\geqslant\theta_0$, we have $\theta(u_1^0,u_2^0,\dots,u_m^0)\geqslant \frac{1}{2}\theta_0$.} Actually,
\begin{eqnarray*}
&\|\sum_i\alpha_iu_i^0\|_H^2=\|\sum_i\alpha_i\tilde{u}_i^0+\sum_i\alpha_i(u_i^0-\tilde{u}_i^0)\|_H\geqslant \frac{3}{4}\|\sum_i\alpha_i\tilde{u}_i^0\|_H^2-3\|\sum_i\alpha_i(u_i^0-\tilde{u}_i^0)\|_H^2 \qquad
\\
&\geqslant \frac{3}{4}\|\sum_i\alpha_i\tilde{u}_i^0\|_H^2-C(m)\max_{i=1,\dots,m}\|(u_i^0-\tilde{u}_i^0)\|_H^2\sum_i\alpha_i^2 \geqslant \frac{3}{4}\|\sum_i\alpha_i\tilde{u}_i^0\|_H^2-C(m)\max_{i=1,\dots,m}\|u_i^0-P_0u_i^0\|\sum_i\alpha_i^2.
\end{eqnarray*}
Namely,
$$
\frac{\|\sum_i\alpha_iu_i^0\|_H^2}{\sum_i\alpha_i^2}\geqslant \frac{3}{4}\frac{\|\sum_i\alpha_i\tilde{u}_i^0\|_H^2}{\sum_i\alpha_i^2}-C(m)\delta(M(\mu),G_0)\geqslant \frac{1}{2}\theta_0,\ \ \mbox{for}\ \delta(M(\mu),G_0)\ \mbox{small enough}.
$$
Therefore, by Lemma \ref{lem:scdist},
\begin{eqnarray}
&\delta(M(\mu),\widetilde{G}_1)\leqslant \sqrt{4m\theta_0^{-1}}\max_{1\leqslant j\leqslant m}\{\delta({\rm span}\{u_j^0\},\widetilde{G}_1)\}
\leqslant \sqrt{4m\theta_0^{-1}}\max_{1\leqslant j\leqslant m}\{\|u_j^0-\alpha_j^0\hat{u}_j^1\|_H\} \nonumber\\
&\leqslant  \sqrt{4m\theta_0^{-1}}(C_{\mu,2}\|T-T\tilde P_0\|_H\delta(M(\mu),\widetilde{G}_0) + \delta(M(\mu),G_1)),
\end{eqnarray}
Similarly, we can obtain that
\begin{equation}
\delta(M(\mu),\widetilde{G}_{l+1}) \leqslant \sqrt{4m\theta_0^{-1}}( C_{\mu,2}\|T-T\widetilde P_l\|_H\delta(M(\mu),\widetilde{G}_l) + \delta(M(\mu),G_{l+1})),\ \ l=1,2,\dots,N-1.
\end{equation}
Therefore,
\begin{equation}
\delta(M(\mu),\widetilde{G}_N)\leqslant \sqrt{4m\theta_0^{-1}}\sum_{l=0}^{N}(\sqrt{4m\theta_0^{-1}}C_{\mu,2})^{(N-l)}[\prod_{j=l}^{N-1}\|T-T\widetilde{P}_{j}\|_H]\delta(M(\mu),G_l).
\end{equation}
%Noting that $G_0\subset\tilde{G}_j$, and therefore $\tilde{P}_jP_0v=P_0v$ for any $v\in H$, we have
%$$
%I-\tilde{P}_j=(I-\tilde{P}_j)(I-P_0)
%$$
%$$
%I-P_0=(I-P_j)(I-P_0)=I-P_j-P_0+
%$$
%$$
%T(I-\tilde{P}_j)=T(I-\tilde{P})
%$$
%$$
%(T-\tilde{T}_j)v=(I-\tilde{P}_j)Tv=(I-\tilde{P}_j)(I-P_0)Tv=(I-\tilde{P}_j)(T-T_0)v.
%$$
%Therefore, $\|(T-\tilde{T}_j)v\|_H\leqslant C\|(T-T_0)v\|_H$, and $\|T-\tilde{T}_j\|_H\leqslant \|T-T_0\|_H$.
%Then \eqref{thm:conves} follows immediately by Lemma \ref{lem:hbthm7.1} again and the assumptions. Simultaneous \eqref{thm:convev} follows. The proof is completed.
The proof is completed by setting $\beta_1=\sqrt{4m\theta_0^{-1}}$ and $\beta_2=\sqrt{4m\theta_0^{-1}}C_{\mu_2}$.
\end{proof}

\begin{remark}
By Lemma \ref{lem:delta}, Lemma  \ref{lem:hbthm7.1}, it follows for $G_0$ big enough that
\begin{equation}
\hat{\delta}(M(\mu),\widetilde{M}_N(\mu))\leqslant \beta_1\sum_{l=0}^{N}[\prod_{j=l}^{N-1}(\beta_2\|T-T\widetilde{P}_{j}\|_H)]\delta(M(\mu),G_l).
\end{equation}
\end{remark}

\begin{remark}
If $T_0$ is a good approximation of $T$, then $|\alpha_j^0-1|$ is small. Therefore, if we modify the algorithm in {\bf Step 1.i.3} by replacing $\hat{u}_j^i$ with some $\breve{u}_j^i\in G_i$ such that $\|\hat{u}_j^i-\breve{u}_j^i\|\leqslant C\delta(M(\mu),G_1)$ for some constant $C$, then the result of the lemma keeps true.
\end{remark}

\subsection{Spectral approximation of generalized symmetric operator}
Let $a(\cdot,\cdot)$ be a bounded symmetric bilinear form defined on the Hilbert space $H$.

\begin{definition}
If for any $u\in H$, there is a unique $v\in H$, such that
$$
a(w,v)=a(Sw,u),\ \forall\,w\in H,
$$
then define $S^{a*}:H\to H$, the adjoint operator of $S$ with respect to $a(\cdot,\cdot)$, by $S^{a*}u:=v$. If for an operator $S:H\to H$, $S^{a*}$ exists and $S^{a*}=S$, then $S$ is called symmetric with respect to $a(\cdot,\cdot)$, or $a(\cdot,\cdot)$-symmetric.
\end{definition}

\begin{lemma}
If both $R^{a*}$ and $S^{a*}$ exist, then $(R\circ S)^{a*}$ exists, and $(R\circ S)^{a*}=S^{a*}\circ R^{a*}$.
\end{lemma}

We propose the hypothesis below for an operator $S$.
\paragraph{\textbf{Hypothesis HC}} For any $u\in H$, $a(Su,u)=0$ if and only if $\|Su\|_H=0$.

\begin{lemma}
Let $S$ be $a(\cdot,\cdot)$-symmetric and satisfy {\bf HC}, then the eigenvalues of $S$ are all real.
\end{lemma}
\begin{proof}
Let $H$, $S$, and $a(\cdot,\cdot)$ be complexified in the usual manner. Let $\lambda\neq 0$ be an eigenvalue of $S$, and $u$ be an eigenvector that belongs to $\lambda$. Then
$$
|\lambda-\bar{\lambda}||a(u,u)|=|a((S-\lambda I)u,u)-a((S-\bar{\lambda}I)u,u)|=|a((S-\lambda I)u,u)-a(u,(S-\lambda I)u)|=0.
$$
Namely $\lambda-\bar{\lambda}=0$. This finishes the proof.
\end{proof}

\begin{lemma}
Let $S$ be $a(\cdot,\cdot)$-symmetric and satisfy {\bf HC}. Let $\mu_1\neq\mu_2$ be two distinct eigenvalues of $S$. Then
\begin{equation}
a(u,v)=0,\ \ \forall\,u\in M(\mu_1),\ v\in M(\mu_2).
\end{equation}
\end{lemma}
\begin{proof}
Without loss of generality, assume $\mu_1\neq 0$, then
\begin{equation}
a(u,v)=\mu_1^{-1}a(\mu_1 u,v)=\mu_1^{-1}a(Su,v)=\mu_1^{-1}a(u,Sv)=(\mu_2/\mu_1) a(u,v).
\end{equation}
Since $\mu_1\neq\mu_2$, it follows that $a(u,v)=0$. This finishes the proof.
\end{proof}

\begin{lemma}\label{lem:samesign}
Let $S$ be $a(\cdot,\cdot)$-symmetric and satisfy {\bf HC}, then all $\mu a(u,u)$ take the same sign, where $u$ is an eigenvector of $S$ that belongs to $\mu$, a nonzero eigenvalue of $S$.
\end{lemma}
\begin{proof}
Let $\dim(M(\mu))=m$, $\mu\neq 0$, then by Gram-Schmidt process, there exist $m$ linearly independent eigenvectors $\{u_j\}$, such that $a(u_i,u_j)=0$, for $1\leqslant i\neq j\leqslant m$. Now given $u=\sum_i\alpha_iu_i$, $a(u,u)=\sum_{i=1}^m\alpha_i^2a(u_i,u_i)$. Since $a(u,u)\neq 0$, we have all $a(u_i,u_i)$ take the same sign, $i=1,\dots,m$. We can set $a(u,u)>0$. Then there are two constants $0<c_s<c_b$, such that
\begin{equation}\label{eq:equina}% equivalence between norm and a(\cdot,\cdot)
c_s\|v\|_H^2\leqslant a(v,v)\leqslant c_b\|v\|_H^2,\ \ \forall\,v\in M(\mu).
\end{equation}

Now, without loss of generality, given $u,v$ two eigenvectors of $S$ belonging to $\mu$ and $\nu$ respectively, such that $\|Su\|_H\| S v\|_H\neq 0$ and $a(u,v)=0$. Then $a(S(\alpha u+\beta v),\alpha u+\beta v)=\mu\alpha^2a(u,u)+\nu\beta^2a(v,v)$. Thus by {\bf HC}, $\mu a(u,u)$ and $\nu a(v,v)$ take the same sign. The proof is completed.
\end{proof}

Lemmas \ref{lem:listasT} and \ref{lem:listasTdis} then follows from the theory of spectral approximation of compact operators.
\begin{lemma}\label{lem:listasT}
If $T$ is a compact operator and $a(\cdot,\cdot)$-symmetric, then all its eigenvalues are real. Further, if all $a(w,w)$, where $w$ is any eigenvector of $T$ that belongs to some nonzero eigenvalue, take the same sign, the eigenvalues of $T$ can be listed in a sequence as, counting multiplicities and up to the sign,
\begin{equation}
\mu_1\geqslant  \mu_2 \geqslant \mu_3 \geqslant  \mu_4 \geqslant  \dots \geqslant  0.
\end{equation}
\end{lemma}

\begin{lemma}\label{lem:listasTdis}
Let $T$ be a compact operator which is $a(\cdot,\cdot)$-symmetric, and $\{T_h\}_{h>0}$ be a family of compact operators which are $a(\cdot,\cdot)$-symmetric. For each $T_h$, if all $a(w_h,w_h)$, where $w_h$ is any eigenvector of $T_h$ that belongs to some nonzero eigenvalue, take the same sign, its eigenvalues are listed in a sequence as, counting multiplicities and up to the sign,
\begin{equation}
\mu_{1,h}\geqslant  \mu_{2,h} \geqslant \mu_{3,h} \geqslant \mu_{4,h} \geqslant  \dots\geqslant 0 .
\end{equation}
Assume that $T_h$ converges to $T$ in norm as $h\to 0$. Then
\begin{equation}
\lim_{h\to 0}\mu_{k,h} = \mu_k,\ \ k=1,2,\dots.
\end{equation}
\end{lemma}
\begin{remark}
The assumption that all $a(w,w)$ take the same sign where $w$ is any eigenvector of $T$ that belongs to some nonzero eigenvalue is a mild one for elliptic problems, and, according to Brezzi's theory, many types of saddle-point problems. 
\end{remark}

\subsubsection{Spectral approximation by the aid of projection operator}
\begin{lemma}%{\color{red}???}
Let $P_G$ be a projection on $G\subset H$.  If both $T$ and $P_G$ are $a(\cdot,\cdot)$-symmetric on $H$, then $T_G=P_GT$ is $a(\cdot,\cdot)$-symmetric on $G$.
\end{lemma}
\begin{proof}
Given $u,v\in G$, $$a(T_Gu,v)=a(P_GTu,v)=a(u,TP_Gv)=a(P_Gu,Tv)=a(u,P_GTv)=a(u,T_Gv).$$ This completes the proof.
\end{proof}

Let $\mathcal{G}:=\{G_h\}_{h>0}$ be a family of subspaces of $H$, and $P_h$ be the projection operators on $G_h$. Assume that
\begin{equation}\label{eq:aaoa}
\displaystyle\inf_{G_h\in\mathcal{G}}\inf_{u\in H\setminus G_h,v\in G_h}\frac{\|u-v\|_H}{\|u-P_hu\|_H}\geqslant C_0.
\end{equation}
\begin{lemma}\label{lem:estgsev}
Let $T$ and $P_h$ be $a(\cdot,\cdot)$-symmetric, and $T_h=P_hT$ converges to $T$ in norm. Let $\mu$ be a nonzero eigenvalue of $T$ with algebraic multiplicity $m$ and let $\mu_h$ be an eigenvalue of $T_h$ that converge to $\mu$. There is a constant $C$, such that for $h$ sufficiently small,
$|\mu-\mu_h|\leqslant C\hat{\delta}(M(\mu),M_h(\mu))^2$.
\end{lemma}
\begin{proof}
Firstly, let $v_h\in G_h$, then $$a(T_hv_h,v_h)=a(P_hTv_h,v_h)=a(Tv_h,P_h^{a*}v_h)=a(v_h,T^{a*}P_h^{a*}v_h)=a(v_h,TP_hv_h)=a(Tv_h,v_h).$$

Let $u_h\in M_h(\mu)$, with $\|u_h\|_H=1$. There is $u\in M(\mu)$, $\|u\|_H=1$, such that $u_h= \gamma E_hu$. Note that $a(Tu,v)=\mu a(u,v)$, and $a(T_hu_h,v_h)=\mu_ha(u_h,v_h)$, for $v\in H$ and $v_h\in G_h$. Then
\begin{eqnarray*}
&|\mu-\mu_h|a(u_h,u_h)=|a(T(u-u_h),(u-u_h))-\mu a((u-u_h),(u-u_h))|
\\
&\leqslant C\|u-u_h\|_H^2 \leqslant C(\|u-E_hu\|_H^2+|\gamma-1|^2\|E_hu\|^2).
\end{eqnarray*}
By Lemma \ref{lem:proj} and \eqref{eq:aaoa}, we can prove $\|u-E_hu\|_H\leqslant \hat{\delta}(M(\mu),M_h(\mu))$ and $|\gamma-1|\leqslant \hat{\delta}(M(\mu),M_h(\mu))$. The proof is then completed by noting that, by \eqref{eq:aaoa}, Lemma \ref{lem:proj} and  \eqref{eq:equina}, we have $c_{sh}\|v_h\|_H^2\leqslant a(v_h,v_h)\leqslant c_{bh}\|v_h\|_H^2$ for $v_h\in M_h(\mu)$.
\end{proof}

\subsection{Variational formulation}

Let $H$ be a Hilbert space, and $a(\cdot,\cdot)$ and $b(\cdot,\cdot)$ be two bounded symmetric bilinear forms on $H$. Besides, $b(u,u)\geqslant 0$ for $u\in H$. Let an operator $T:H\to H$ be defined by
$$
a(Tw,v)=b(w,v),\ \ \forall\,v\in H.
$$

\paragraph{\textbf{Hypothesis\ HIS}} $\displaystyle \inf_{v\in H}\sup_{w\in H}\frac{a(v,w)}{\|v\|_H\|w\|_H}\geqslant C$.

\begin{lemma}
If $a(\cdot,\cdot)$ satisfies {\bf HIS}, then,
\begin{enumerate}
\item $T$ is uniquely defined, and, $\|T\|_H\cequiv\|T^{a*}\|_H$\footnote{From this point onwards, $\lesssim$, $\gtrsim$, and $\cequiv$ respectively denote $\leqslant$, $\geqslant$, and $=$ up to a constant. The hidden constants depend on the domain, and, when triangulation is involved, they also depend on the shape-regularity of the triangulation, but they do not depend on $h$ or any other mesh parameter.};
\item $T$ is $a(\cdot,\cdot)$-symmetric, and {\bf HC} holds.
\end{enumerate}
\end{lemma}
\begin{proof}
The existence of $T^{a*}$ follows from the Babu\v{s}ka theory. Moreover, we have  $\displaystyle\|u\|_H\cequiv \sup_{v\in H}\frac{a(u,v)}{\|v\|}.$ Therefore,
$\displaystyle\|T\|_H\cequiv\sup_{v\in H}\sup_{w\in H}\frac{a(Tv,w)}{\|w\|_H\|v\|_H}=\sup_{v\in H}\sup_{w\in H}\frac{a(v,T^{a*}w)}{\|w\|_H\|v\|_H}\cequiv\|T^{a*}\|_H$.

The $a(\cdot,\cdot)$-symmetry follows from the definition. Define $B:H\to H$ by $(Bv,w)_H=b(v,w)$, where $(\cdot,\cdot)_H$ is the basic inner product equipped onto $H$. Then $B$ is uniquely defined, and $B$ is self-adjoint. Note that $b(u,u)\geqslant 0$, and we have $B$ positive semi-definite. Particularly, it is easy to show that $(Bv,v)_H=0$ if and only if $Bv=0$. Namely, $b(u,u)=0$ if and only if $b(u,v)=0$ for any $v\in H$. Further, $a(Tu,u)=0$ if and only if $a(Tu,v)=0$ for any $v\in H$, which by {\bf HIS} is equivalent to $Tu=0$. Thus {\bf HC} holds. The proof is completed.
\end{proof}
\begin{remark}
In general, $T$ can not be symmetric with respect to the intrinsic inner product of $H$.
\end{remark}

%{\color{red}Since the assumptions {\bf HC}, {\bf H1} and {\bf H2} hold???????}, $T$ is well defined, and $T$ is $a(\cdot,\cdot)$-symmetric. Assume $T$ is compact on $H$, and then the theory constructed in previous sections works.
%
%\begin{lemma}
%Provided \textrm{H1} and \textrm{H2}.
%\begin{enumerate}
%\item For any $T:H\to H$, its $a(\cdot,\cdot)$-adjoint operator exists.
%\item $\|T\|_H\cequiv\|T^{a*}\|_H$.
%\end{enumerate}
%\end{lemma}
%\begin{proof}
%Provided \textrm{H1} and \textrm{H2}, the existence of $T^{a*}$ follows from the Babushka theory. Moreover, we have  $\displaystyle\|u\|_H\cequiv \sup_{v\in H}\frac{a(u,v)}{\|v\|}.$ Therefore,
%$\displaystyle\|T\|_H\cequiv\sup_{v\in H}\sup_{w\in H}\frac{a(Tv,w)}{\|w\|_H\|v\|_H}=\sup_{v\in H}\sup_{w\in H}\frac{a(v,T^{a*}w)}{\|w\|_H\|v\|_H}\cequiv\|T^{a*}\|_H$. The proof is completed.
%\end{proof}
%
%

Let $\mathcal{G}:=\{G_i\}_{i=0,1,\dots}$ be such that
$$
G_0\subset G_1\subset \dots\subset H.
$$
Define operators $P_i:H\to G_i$ and $T_i:H\to G_i$ by
$$
a(P_iw,v)=a(w,v),\ \ w\in H,\forall\,v\in H,\quad a(T_iw,v)=b(w,v),\ \ \forall\,v\in G_i.
$$

\paragraph{\textbf{Hypothesis HISG}} $\displaystyle\inf_{G\in \mathcal{G}}\inf_{v\in G}\sup_{w\in G}\frac{a(v,w)}{\|v\|_H\|w\|_H}\geqslant C'$.\\
The lemma below is standard.
\begin{lemma}\label{lem:oa}
If $a(\cdot,\cdot)$ and $\mathcal{G}$ satisfy {\bf HIS} and {\bf HISG}, the two operators $P_i$ and $T_i$ are well defined. Evidently, $T_i=P_iT$. Besides,
\begin{equation}
\|(I-P_G)w\|_H\leqslant (1+\frac{1}{C}+\frac{1}{C'})\inf_{v\in G}\|w-v\|_H,\ \ \forall\,w\in H.
\end{equation}
\end{lemma}

%Assume $\|u-P_iu\|_H\to 0$ as $i\to \infty$ for any $u\in H$. Then $\{T_i\}$ is a convergent approximation of $T$.

\begin{lemma}\label{lem:listTevvf}
Provided the assumptions of Lemmas \ref{lem:listasT} and \ref{lem:listasTdis}. Let the eigenvalues of $T$ be listed in a sequence as, counting multiplicities,
\begin{equation}
\mu_1\geqslant \mu_2 \geqslant \mu_3 \geqslant \mu_4 \geqslant \dots \geqslant 0.
\end{equation}
For each $T_i$, list its eigenvalues in a sequence as
\begin{equation}
\mu_{1,i}\geqslant \mu_{2,i} \geqslant \mu_{3,i} \geqslant \mu_{4,i} \geqslant \dots .\geqslant \mu_{N_i,i}\geqslant 0.
\end{equation}
Provided $P_iu\to u$ for $u\in H$, then
\begin{equation}
\lim_{i\to \infty}\mu_{k,i} = \mu_k,\ \ k=1,2,\dots.
\end{equation}
\end{lemma}

\subsubsection{Multi-level algorithm in variational form}

\begin{algorithm}\label{alg:variational}
An N-level algorithm for first $k$ eigenvalues of $T$.

%\fbox{
%\begin{minipage}{1.\textwidth}

\begin{description}
\item[Step 0] Construct a series of nested spaces $G_0\subset G_1\subset\dots\subset G_N\subset H$. Set $\widetilde{G}_0=G_0$.
\item[Step 1] For $i=1:1: N$, generate auxiliary space triples $\widetilde{G}_i$ recursively.
\begin{description}
\item[Step 1.i.1]
Solve the eigenvalue problem below for its first $k$ eigenpairs $(\tilde{\mu}_j^{i-1},\tilde{u}_j^{i-1})_{j=1,\dots,k}$
$$
\mu a(\tilde{u},v)=b(\tilde{u},v),\tilde{u}\in \widetilde{G}_{i-1},\ \ \forall\,v\in \widetilde{G}_{i-1},
$$
such that $a(\tilde{u}_j^{i-1},\tilde{u}_l^{i-1})=0$, for $1\leqslant j\neq l\leqslant k$.
\item[Step 1.i.2] Compute
$$
a(\hat{u}_j^i,v)=\frac{1}{\tilde{\mu}_j^{i-1}}b(\tilde{u}_j^{i-1},v),\ \ \forall\,v\in G_i.
$$
\item[Step 1.i.3] Set
$$
\widetilde{G}_i=G_0+{\rm span}\{\hat{u}_j^i\}_{j=1}^k.
$$
\end{description}
\item[Step 2] Solve eigenvalue problem for its first $k$ eigenpairs $(\tilde{\mu}_j^N,\tilde{u}_j^N)_{j=1,\dots,k}$:
$$
\mu a(\tilde{u},v)=b(\tilde{u},v),\tilde{u}\in \widetilde{G}_{N},\ \ \forall\,v\in \widetilde{G}_N.
$$
such that $a(\tilde{u}_j^N,\tilde{u}_l^N)=0$, for $1\leqslant j\neq l\leqslant k$.
\end{description}

%\end{minipage}
%}
\end{algorithm}

%\begin{lemma}\label{lem:equivnorm}
%Provided {\bf HIS}, if both $T$ and $P_G$ are $a(\cdot,\cdot)$-symmetric, $\|T(I-P_G)\|_H\cequiv \|(I-P_G)T\|_H$.
%\end{lemma}
\begin{lemma}\label{lem:nearlyothgonal}
Let $\mu$ be a nonzero eigenvalue of $T$, with multiplicity $m$, and $M(\mu)$ the eigenspace. Let $\{T_h\}$ be a family of approximating operators, and $\mu_{1,h},\dots,\mu_{m,h}$ be the eigenvalues of $T_h$ approximating $\mu$. Let $\{u_{i,h}\}$ be the unit eigenvectors with respect to $\mu_{i,h}$, such that $a(u_{i,h},u_{j,h})=0$ for $1\leqslant i\neq j\leqslant m$. There is a constant $c$, such that $\theta(u_{1,h},\dots,u_{m,h})\geqslant c$ for $h$ sufficiently small.
\end{lemma}
\begin{proof}
Firstly, there are two constants $0<c_s<c_b$, such that
$$
c_s\|v\|_H^2\leqslant a(v,v)\leqslant c_b\|v\|_H^2,\ \ \forall\,v\in M(\mu).
$$
Therefore, there are two constants $0<c_s'<c_b'$, such that for $h$ sufficiently small,
$$
c_s'\|u_{h,i}\|_H^2\leqslant a(u_{h,i},u_{h,i})\leqslant c_b'\|u_{h,i}\|_H^2,\ \ 1\leqslant i\leqslant m,
$$
and further, with $0<c_s''<c_b''$,
$$
c_s''\|v_{h}\|_H^2\leqslant a(v_{h},v_{h})\leqslant c_b''\|v_{h}\|_H^2,\ \ \forall\,v_h\in M_h(\mu).
$$
Now, given $u_h=\sum_i\beta_iu_{i,h}$, then
\begin{eqnarray}
&\sum_{i}\beta_i^2\|u_{i,h}\|_H^2\leqslant c_s''^{-1}\sum_{i}\beta_i^2a(u_{i,h},u_{i,h}) = c_s''^{-1}\sum_{i}a(\beta_iu_{i,h},\beta_iu_{i,h})
\nonumber\\
&= c_s''^{-1}a(\sum_{i}\beta_iu_{i,h},\sum_{i}\beta_iu_{i,h})\leqslant c_b''/c_s''\|\sum_i\beta_iu_{i,h}\|_H^2.
\end{eqnarray}
The proof is completed by the definition of $\theta(u_{1,h},\dots,u_{m,h})$.
\end{proof}

\begin{theorem} \label{thm:cmlavf}% Convergence of multi-level Algorithm
There exist constants $\beta_1$ and $\beta_2$ dependent of $\mu$, such that, with $G_0$ big enough,
\begin{equation}
\delta(M(\mu),\widetilde{G}_N)\leqslant \beta_1\sum_{l=0}^{N}(\beta_2\|T-TP_0\|_H)^{N-l}\delta(M(\mu),G_l).
\end{equation}
\end{theorem}
\begin{proof}
Since $G_0\subset \widetilde{G}_j$, $(I-P_0)(I-\widetilde{P}_j)=I-\widetilde{P}_j$, and $\|T-T\widetilde{P}_j\|_H=\|T(I-P_0)(I-\widetilde{P}_j)\|_H\|_H\leqslant \|T-TP_0\|_H$. The result then follows from Lemma \ref{lem:nearlyothgonal} and Theorem \ref{thm:cmla}.
\end{proof}

%
%The condition and conclusion of Theorem \ref{thm:cmla} work in this situation.
%
%
%\begin{lemma}
%Let
%
%Define for $u\in M(\mu)$ that $|||u|||=\sqrt{a(u,u)}$. Then
%$$
%c_\mu\|u\|_H\leqslant |||u|||_H\leqslant C_\mu\|u\|_H.
%$$
%\end{lemma}
%
%\begin{lemma}
%Assume {\bf H1} and {\bf H2} are true. Let $\{\varphi_i\}_{i=1:m}$ be unit vectors in $H$, and $a(\varphi_i,\varphi_j)=0$ for $i,j$. Then there is a constant $C=C(m)$, such that $\theta(\varphi_1,\dots,\varphi_m)>C_m$.
%\end{lemma}
%\begin{proof}
%{\color{red} proofproof}
%\end{proof}
%
%\begin{proof}
%
%\end{proof}
%

%
%
%
\section{Mixed method for the biharmonic eigenvalue problem}
\label{sec:mm} % mm = Mixed Method

In this section, we present a mixed method for the biharmonic eigenvalue problem. We will first construct an equivalent mixed formulation of the eigenvalue problem (Theorem \ref{thm:evpequiv}), and then consider its direct discretization (Theorem \ref{thm:singlelevel}) and multi-level scheme (Theorem \ref{thm:bhmulti-level}) within the framework presented in Section \ref{sec:sa}. The optimal complexity of the algorithm is also discussed. 
\subsection{Preliminary theory of eigenvalue problem}
Let $\Omega\subset\mathbb{R}^2$ be a polygonal domain, and $\Gamma=\partial\Omega$ be the boundary of $\Omega$. Let $H^1(\Omega)$, $H^1_0(\Omega)$, $H^2(\Omega)$, and $H^2_0(\Omega)$ be the standard Sobolev spaces as usual, and $L^2_0(\Omega):=\{w\in L^2(\Omega):\int_\Omega w\dx=0\}$. In this paper, we use the subscript ``$\undertilde{~}$" to denote vector, and particularly, $\undertilde{H}{}^1_0(\Omega)=(H^1_0(\Omega))^2$. Consider the biharmonic eigenvalue problem:
\begin{equation}
\left\{
\begin{array}{rcll}
\displaystyle
\Delta^2u&=&\lambda u &\mbox{in}\,\Omega
\\
\displaystyle u&=& 0 &\mbox{on}\,\partial\Omega,
\\
\displaystyle \frac{\partial u}{\partial n} & = & 0 &\mbox{on}\,\partial\Omega.
\end{array}
\right.
\end{equation}
The variational form is to find $(\lambda,u)\in \mathbb{R}\times H^2_0(\Omega)$, such that
\begin{equation}\label{eq:orieig}
\int_\Omega\nabla^2u:\nabla^2v:=\int_\Omega \sum_{i,j=1}^2\frac{\partial^2 u}{\partial x_i\partial x_j}\frac{\partial^2 v}{\partial x_i\partial x_j}=\lambda(u,v):=\lambda\int_\Omega uv,\ \ \forall\,v\in H^2_0(\Omega).
\end{equation}
By the property of elliptic operators, the problem \eqref{eq:orieig} has an eigenvalue sequence $\lambda_j$:
\begin{equation}
0< \lambda_1\leqslant \lambda_2\leqslant \cdots \leqslant \lambda_k\leqslant \cdots,\ \ \ \mbox{and}\ \ \lim_{k\to \infty}\lambda_k=\infty.
\end{equation}

\subsection{Mixed formulation}

To reduce the order of the Sobolev spaces involved, we begin with the following well known result on the exactness among $H^2_0(\Omega)$, $\undertilde{H}{}^1_0(\Omega)$, and operators $\rot$ and $\nabla$.
\begin{lemma} (\cite{Girault;Raviart1986,FengZhang})
$\nabla H^2_0(\Omega)=\{\upsi\in \undertilde{H}{}^1_0(\Omega):\rot\upsi=0\}$.
\end{lemma}
Define $V:=H^1_0(\Omega)\times \undertilde{H}{}^1_0(\Omega)\times L^2_0(\Omega)\times H^1_0(\Omega)$. Now we can introduce the mixed formulation of the eigenvalue problem: find $(u,\uphi,p,w)\in V$,  such that
\begin{equation}\label{eq:saddleevp}
\left\{
\begin{array}{cccccll}
&&&(\nabla w,\nabla v) & = &\lambda(u,v) & \forall\,v\in H^1_0(\Omega)
\\
& (\nabla \uphi,\nabla \upsi) & +(p,\rot \upsi) & + (\nabla w,\upsi) &=&0 & \forall\,\upsi\in \undertilde{H}{}^1_0(\Omega)
\\
& (\rot\uphi,q) & &&=&0 & \forall\,q\in L^2_0(\Omega)
\\
(\nabla u,\nabla s)& + (\uphi,\nabla s) &&& = &0 &\forall\, s\in H^1_0(\Omega).
\end{array}
\right.
\end{equation}

\begin{theorem}\label{thm:evpequiv}
The eigenvalue problem \eqref{eq:saddleevp} is equivalent to \eqref{eq:orieig}.
\end{theorem}
We postpone the proof of \ref{thm:evpequiv} after some technical results. First, equip $V$ with the norm
$$
\|(u,\uphi,p,w)\|_V:=\left(\|u\|_{1,\Omega}^2+\|\uphi\|_{1,\Omega}^2+\|p\|_{0,\Omega}^2+\|w\|_{1,\Omega}^2\right)^{1/2},
$$
then $V$ is a Hilbert space. Define on $V$ a bilinear form
%\begin{eqnarray}
\begin{multline}
a((u,\uphi,p,w),(s,\upsi,q,v))
\\%\nonumber\\
:=(\nabla w,\nabla v)+ (\nabla \uphi,\nabla \upsi) +(p,\rot \upsi) + (\nabla w,\upsi)+(\rot\uphi,q)+(\nabla u,\nabla s) + (\uphi,\nabla s).
\end{multline}
%\end{eqnarray}

\begin{lemma}\label{eq:contiiso}
Given $F\in V'$, there exists a unique $(u,\uphi,p,w)\in V$, such that
\begin{equation}
a((u,\uphi,p,w),(s,\upsi,q,v))=\langle F,(s,\upsi,q,v)\rangle,\ \ \forall\,(s,\upsi,q,v)\in  V.
\end{equation}
Moreover,
$$
\|(u,\uphi,p,w)\|_V\cequiv \|F\|_{V'}.
$$
\end{lemma}
\begin{proof}
Denote $\hat{a}((u,\uphi),(v,\upsi)):=(\nabla\uphi,\nabla\upsi),$ and $\hat{b}((u,\uphi),(q,s)):=(\rot\uphi,q)+(\nabla u,\nabla s)+(\uphi,\nabla s).$ Accordingly, denote $Z:=\{(u,\uphi)\in H^1_0(\Omega)\times \undertilde{H}{}^1_0(\Omega):\hat{b}((u,\uphi),(q,s))=0\}$. Evidently $\hat{a}(\cdot,\cdot)$ is coercive on $Z$. For any $(q,s)\in L^2_0(\Omega)\times H^1_0(\Omega)$, we can choose $\uphi\in\undertilde{H}{}^1_0(\Omega)$, such that $(\rot\uphi,q)=\|q\|^2_0$, and $\|\uphi\|_{1,\Omega}\leqslant C\|q\|_{0,\Omega}$. Now, let $s_{\uphi}\in H^1_0$ be defined such that $(\nabla s_{\uphi},\nabla v)=(\uphi,\nabla v)$ for any $v\in H^1_0(\Omega)$, and set $u=s-s_{\uphi}$, then $\hat{b}((u,\uphi),(q,s))=\|q\|_{0,\Omega}^2+\|\nabla s\|_{0,\Omega}^2$, and $\|\uphi\|_{1,\Omega}+\|u\|_{1,\Omega}\leqslant C(\|q\|_{0,\Omega}+\|s\|_{1,\Omega})$. This indeed shows the inf-sup condition
\begin{equation}
\inf_{(q,s)\in L^2_0(\Omega)\times H^1_0(\Omega)}\sup_{(u,\uphi)\in H^1_0(\Omega)\times \undertilde{H}{}^1_0(\Omega)}\frac{\hat{b}((u,\uphi),(q,s))}{(\|q\|_{0,\Omega}+\|s\|_{1,\Omega})(\|\uphi\|_{1,\Omega}+\|u\|_{1,\Omega})}\geqslant C.
\end{equation}
The proof is completed by Brezzi's theory.
\end{proof}
\begin{remark}
The inf-sup condition follows immediately.
\begin{equation}
\inf_{(u,\uphi,p,w)\in V}\sup_{(s,\upsi,q,v)\in V}\frac{a((u,\uphi,p,w),(s,\upsi,q,v))}{\|(u,\uphi,p,w)\|_V\|(s,\upsi,q,v)\|_V}\geqslant C.
\end{equation}
\end{remark}
\paragraph{\bf Proof of Theorem \ref{thm:evpequiv}} Given $f\in L^2$, there is a unique $u\in H^2_0(\Omega)$, such that $(\nabla^2u,\nabla^2v)=(f,v)$ for $v\in H^2_0(\Omega)$, and a unique $(\tilde{u},\tilde{\uphi},\tilde{p},\tilde w)\in V$, such that $a((\tilde{u},\tilde{\uphi},\tilde{p},\tilde w),(s,\upsi,q,v))=(f,v)$ for $\forall\,(s,\upsi,q,v)\in V$, and moreover, $\tilde{u}=u$. Now let $(\lambda,u)$ be an eigenpair of \eqref{eq:orieig}, then there is $(\tilde{u},\tilde{\uphi},\tilde{p},w)\in V$, such that $a((\tilde{u},\tilde{\uphi},\tilde{p},\tilde w),(s,\upsi,q,v))=\lambda (u,v)$ for $\forall\,(s,\upsi,q,v)\in V$, and moreover $\tilde{u}=u$. On the other hand, let $(\tilde\lambda, \tilde u,\tilde \uphi,\tilde p,\tilde w)$ be an eigenpair of \eqref{eq:saddleevp}, then there is a unique $u\in H^2_0(\Omega)$, such that $(\nabla^2u,\nabla^2v)=\tilde{\lambda}(\tilde u,v)$, $\forall\,v\in H^2_0(\Omega)$. It follows further that $u=\tilde{u}$. The proof is completed.
\qed

In the sequel, we focus ourselves on \eqref{eq:saddleevp}. Define on $V$
\begin{equation}
b((u,\uphi,p,w),(s,\upsi,q,v)):=(u,v).
\end{equation}
Both $a(\cdot,\cdot)$ and $b(\cdot,\cdot)$ are symmetric. Then \eqref{eq:saddleevp} is rewritten to: find $(u,\uphi,p,w)\in V$, such that
\begin{equation}\label{eq:saddleevpcf}
a((u,\uphi,p,w),(s,\upsi,q,v)) = \lambda b((u,\uphi,p,w),(s,\upsi,q,v)),\ \ \forall\,(s,\upsi,q,v)\in V.
\end{equation}

Associated with $a(\cdot,\cdot)$ and $b(\cdot,\cdot)$, we define an operator $T$ by
\begin{equation}
a(T(u,\uphi,p,w),(s,\upsi,q,v))=b((u,\uphi,p,w),(s,\upsi,q,v)),\ \forall\, (s,\upsi,q,v)\in V.
\end{equation}
\begin{lemma}
The operator $T$ is well defined from $V$ to $V$, $a(\cdot,\cdot)$-symmetric, and compact. %operator on $H^1_0\times \undertilde{H}{}^1_0(\Omega)\times L^2_0\times H^1_0$.
\end{lemma}
\begin{proof}
The well-posedness of $T$ follows directly from that $a(\cdot,\cdot)$ induces an isomorphism between $V$ and its dual, and $b(\cdot,\cdot)$ is continuous on $V$. As both $a(\cdot,\cdot)$ and $b(\cdot,\cdot)$ are symmetric, $T$ is $a(\cdot,\cdot)$-symmetric. Now, let $\{(u_j,\uphi{}_j,p_j,w_j)\}$ be a bounded sequence in $V$, then there is subsequence $\{(u_{j_k},\uphi{}_{j_k},p_{j_k},w_{j_k})\}$, such that $\{u_{j_k}\}$ is a Cauchy sequence in $L^2(\Omega)$. Therefore,  $\{T(u_{j_k},\uphi{}_{j_k},p_{j_k},w_{j_k})\}$ is a Cauchy sequence in $V$, which, further, has a limit therein. This finishes the proof.
\end{proof}

The eigenvalue problem \eqref{eq:saddleevp} is equivalent to finding $0\neq\mu\in\mathbb{R}$ and $(u,\uphi,p,w)\in V$, such that $T(u,\uphi,p,w)=\mu (u,\uphi,p,w)$, then $\lambda=\frac{1}{\mu}$ and $u$ is the eigenpair we are seeking for.

\begin{remark}\label{rem:vsstokes}
The formulation \eqref{eq:saddleevp} is a saddle-point problem, while the variables $p$ and $w$ can be viewed as two Lagrangian multipliers. However, we note that the right hand side $b(\cdot,\cdot)$ is not coercive on the space of the primal variables ($u$ and $\uphi$) nor on the space of the Lagrangian variables. This makes the classical theory for saddle-point problems, such as discussions in \cite{LinLuoXie2013}, \cite{Mercier.B;Osborn.J;Rappaz.J;Raviart.P1981} or \cite{Boffi.D;Brezzi.F;Gastaldi.L1997}, not directly work for \eqref{eq:saddleevp}. This way, some generalized theory has to be developed.
\end{remark}

\subsection{Discretization and accuracy}

Let $H^1_{h0}$, $\undertilde{H}{}^1_{h0}$ and $L^2_{h0}$ be some specific finite element subspaces of $H^1_{0}$, $\undertilde{H}{}^1_{0}$ and $L^2_{0}$, respectively.  We introduce the discretized mixed eigenvalue problem:
\begin{quote}
find $(u_h,\uphi{}_h,p_h,w_h)\in V_h:=H^1_{h0}\times \undertilde{H}{}^1_{h0}\times L^2_{h0}\times H^1_{h0}$,  such that
\begin{equation}\label{eq:saddleevpdis}
\left\{
\begin{array}{cccccll}
&&&(\nabla w_h,\nabla v_h) & = &\lambda_h(u_h,v_h) & \forall\,v\in H^1_{h0}
\\
& (\nabla \uphi{}_h,\nabla \upsi{}_h) & +(p_h,\rot \upsi{}_h) & + (\nabla w_h,\upsi{}_h) &=&0 & \forall\,\upsi{}_h\in \undertilde{H}{}^1_{h0}
\\
& (\rot\uphi{}_h,q_h) & &&=&0 & \forall\,q_h\in L^2_{h0}
\\
(\nabla u_h,\nabla s_h)& + (\uphi{}_h,\nabla s_h) &&& = &0 &\forall\, s_h\in H^1_{h0}.
\end{array}
\right.
\end{equation}
\end{quote}

For the well-posedness of the discretized problem, we propose the assumption below.
\paragraph{\bf Assumption AIS} The discrete inf-sup condition holds uniformly that
\begin{equation}\label{eq:assmdisis}
\inf_{q_h\in L^2_{h0}}\sup_{\upsi{}_h\in \undertilde{H}{}^1_{h0}}\frac{(\rot\upsi{}_h,q_h)}{\|\nabla_h\upsi{}_h\|_{0,\Omega}\|q_h\|_{0,\Omega}}\geqslant C.
\end{equation}
\begin{remark}
In two dimensional, $\rot$ is the perpendicular of $\nabla$. Considering the homogeneous boundary condition imposed on $\undertilde{H}{}^1_0(\Omega)$, we know that the condition \eqref{eq:assmdisis} is equivalent to the well-known inf-sup condition for the incompressible Stokes problem. 
\end{remark}

\begin{lemma}
Assume the assumption \textbf{AIS} holds. There exists a constant $C$, uniformly with respect to $V_h$, such that \begin{equation}\label{eq:disinfsup}
\inf_{(u_h,\uphi{}_h,p_h,w_h)\in V_h}\sup_{(s_h,\upsi{}_h,q_h,v_h)\in V_h}\frac{a((u_h,\uphi{}_h,p_h,w_h),(s_h,\upsi{}_h,q_h,v_h))}{\|(u_h,\uphi{}_h,p_h,w_h)\|_V\|(s_h,\upsi{}_h,q_h,v_h)\|_V}\geqslant C.
\end{equation}
\end{lemma}
\begin{proof}
The proof is the same as that of Lemma \ref{eq:contiiso}.
\end{proof}

The projection operator $P_h:V\to V_h$ is defined associated with $a(\cdot,\cdot)$ by
\begin{equation}
a(P_h(u,\uphi,p,w),(s_h,\upsi{}_h,q_h,v_h))=a((u,\uphi,p,w),(s_h,\upsi{}_h,q_h,v_h)),\ \ \forall\,(s_h,\upsi{}_h,q_h,v_h)\in V_h.
\end{equation}
By Lemma \ref{lem:oa}, we have the optimal approximation below.
\begin{lemma}\label{lem:aaoavhph}
Given assumption {\bf AIS}, $P_h$ is well defined. There exists a constant $C$, such that
\begin{equation}
\|(u,\uphi,p,w)-P_h(u,\uphi,p,w)\|_V\leqslant C\inf_{(v_h,\upsi{}_h,q_h,s_h)\in V_h}\|(u,\uphi,p,w)- (s_h,\upsi{}_h,q_h,v_h)\|.
\end{equation}
\end{lemma}
%The discrete eigenvalue problem is well defined:
%\begin{multline}\label{eq:disevp}
%\mu_ha((u_h,\uphi{}_h,p_h,w_h),(v_h,\upsi{}_h,q_h,s_h))=b((u_h,\uphi{}_h,p_h,w_h),(v_h,\upsi{}_h,q_h,s_h)),\\ \forall\,(v_h,\upsi{}_h,q_h,s_h)\in V_h.
%\end{multline}

List the eigenvalues of $T$ as
\begin{equation}
\mu_1\geqslant \mu_2\geqslant ... \geqslant 0.
\end{equation}
By Lemma \ref{lem:listasTdis}, the eigenvalues of $T:=P_hT$ can be listed as
\begin{equation}
\mu_{1,h}\geqslant\mu_{2,h}\geqslant ... \geqslant \mu_{N_h,h},
\end{equation}
where $N_h$ is the dimension of $V_h$.
If $V_h$ provides approximation of $V$, namely $(I-P_h)$ tends to zero as $h\to 0$ pointwise, then $\lim_{h\to 0}\mu_{i,h}=\mu_i$,  $i=1,2,\dots.$

Let $\mu$ be a nonzero eigenvalue of $T$ with multiplicity $m$. Denote
$$
M(\mu):=\{(s,\upsi,q,v)\in V:T(s,\upsi,q,v)=\mu (s,\upsi,q,v)\}.
$$
Assume $h$ is sufficiently small, and $\mu_{(1),h},\mu_{(2),h},\dots,\mu_{(m),h}$ be the discrete eigenvalues to approximate $\mu$, and $(u,\uphi,p,w)_{(i),h}$ be the corresponding eigenfunctions. Denote
$$
M_h(\mu):={\rm span}\{(u,\uphi,p,w)_{(i),h}\}_{i=1}^m.
$$
By Lemma \ref{lem:aaoavhph} and Lemma \ref{lem:hbthm7.1}, we have the estimate below.
\begin{lemma}
There exists a constant $C_\mu$, uniform for $h$ sufficiently small, such that
$$
\hat{\delta}(M(\mu),M_h(\mu))\leqslant C_\mu\delta(M(\mu),V_h).
$$
%$$
%|\mu-\mu_{(j),h}|\leqslant C_\mu(\delta(M(\mu),V_h))^2.
%$$
\end{lemma}
Note that $M(\mu)$ and $M_h(\mu)$ coincides with the continuous and discretized spaces $M(\mu^{-1})$ and $M_h(\mu^{-1})$ of \eqref{eq:saddleevp} and \eqref{eq:saddleevpdis}, respectively. We thus have the result below by Lemma \ref{lem:estgsev}.
\begin{theorem}\label{thm:singlelevel}
Let $\lambda$ be the $k$-th eigenvalue of \eqref{eq:saddleevp} $($thus \eqref{eq:orieig}$)$, with $M(\lambda)$ being its invariant subspace; let $(\lambda_h,(u_h,\uphi{}_h,p_h,w_h))$ be the $k$-th eigenpair of \eqref{eq:saddleevpdis}. Then $\lambda_h\to\lambda$ as $h\to 0$. Further, for $h$ sufficiently small,
$$
|\lambda_h-\lambda|\leqslant C\delta(M(\lambda),V_h)^2,
$$
and
$$
\delta((u_h,\uphi{}_h,p_h,w_h), M(\lambda))\leqslant C \delta(M(\lambda),V_h).
$$
Moreover, there exists a $u\in H^2_0(\Omega)$ being an eigenvector of \eqref{eq:orieig} belonging to $\lambda$, such that
$$
\|u_h-u\|_{1,\Omega}\leqslant C\delta(M(\lambda),V_h).
$$
\end{theorem}

\subsubsection{Lagrangian type finite element discretization}
 Directly, we can choose $H^1_{h0}$ to be the $H^1$ Lagrange element space of $k$-th degree, $\widetilde{H}^1_{h0}$ to be the vector $H^1$ Lagrange element space of $k$-th degree, and $L^2_{h0}$ to be the $H^1$ Lagrange element space of $(k-1)$-th degree, $k=2,3,\dots$. We denote this construction by Lagrangian type triple $P_k\sim P_k\sim P_{k-1}$. Similarly, we can choose, e.g., $H^1_{h0}$ to be the $H^1$ Lagrange element space of second degree, $\widetilde{H}^1_{h0}$ to be the vector $H^1$ Lagrange element space of second degree, and $L^2_{h0}$ to be the space of piecewise constants. We denote this choice by reduced Lagrangian type triple $P_2\sim P_2\sim P_0$.

\begin{lemma}
Let $V_h$ be constructed by the Lagrangian type triple $P_k\sim P_k\sim P_{k-1}$, then if $M(\lambda)\subset (H^{k+1}(\Omega)\times \undertilde{H}{}^{k+1}(\Omega)\times H^{k}(\Omega)\times H^{k+1}(\Omega))\cap V$,
$$
\hat{\delta}(M(\mu),M_h(\mu))\leqslant C(M(\mu)) h^{k}, k=2,3,\dots.
$$
Let $V_h$ be constructed by the Lagrangian type triple $P_2\sim P_2\sim P_0$, then if $M(\lambda)\subset (H^{2}(\Omega)\times \undertilde{H}{}^{2}(\Omega)\times H^{1}(\Omega)\times H^{2}(\Omega))\cap V$,
$$
\hat{\delta}(M(\mu),M_h(\mu))\leqslant C(M(\mu)) h.
$$
\end{lemma}

\subsection{Multi-level scheme with Lagrange type elements}

To implement the multi-level algorithm, we construct the multi-level auxiliary spaces on multi-level grids. Let $\mathcal{T}_{h_i}$, $i=0,1,\dots,N$, be a series of nested grids on $\Omega$.  Particularly, we set $h_i\approx \kappa^ih_0$. The spaces $V_{h_i}$ are constructed thereon.

\begin{lemma}\label{lem:deltahhh}
Let $\widetilde{M}_N(\mu)$ be the approximation invariant subspace of $M(\mu)$ generated by Algorithm \ref{alg:variational}.
If there is a constant $C$, such that for $h$ sufficiently small, $\delta(M(\mu),V_h)\leqslant Ch^\tau$, then there is a constant $C'$, such that, for $\mathcal{T}_{h_0}$ sufficiently fine,
$$
\delta(\widetilde{M}_N(\mu),M(\mu))\leqslant C'h^\tau.
$$
\end{lemma}
\begin{proof}
By Theorem \ref{thm:cmlavf},
\begin{equation}
\delta(M(\mu),\widetilde{M}_N(\mu))\leqslant \beta_1\sum_{l=0}^{N}(\beta_2\|T-TP_0\|_H)^{N-l}\delta(M(\mu),V_{h_l})\leqslant \beta_1'\sum_{l=0}^{N}(\beta_2\|T-TP_0\|_H)^{N-l}\kappa^{\tau(l-N)}h_N^\tau.
\end{equation}
Note that in the current context, 
\begin{eqnarray*}
\|T(I-P_h)(u,\uphi,p,w)\|_V\cequiv \sup_{(v,\upsi,q,s)\in V}\frac{a((I-P_h)(u,\uphi,p,w),(s,\upsi,q,v))}{\|(s,\upsi,q,v)\|_V}\\
=\sup_{(v,\upsi,q,s)\in V}\frac{b((I-P_h)(u,\uphi,p,w),(s,\upsi,q,v))}{\|(s,\upsi,q,v)\|_V}.
\end{eqnarray*}
By dual argument, if $\mathcal{T}_{h_0}$ is sufficiently fine, such that $\beta_2\|T-TP_0\|_V/\kappa^\tau<1$, then
$$
\delta(M(\mu),\widetilde{M}_N(\mu))\leqslant \beta_1'h_N^\tau\sum_{l=0}^N(\beta_2\|T-TP_0\|_V/\kappa^\tau)^{N-l}=\frac{\beta_1'}{1-\beta_2\|T-TP_0\|_V/\kappa^\tau}\cdot h_N^\tau.
$$
The proof is finished.
\end{proof}
The theorem below follows immediately.
\begin{theorem}\label{thm:bhmulti-level}
Let $\lambda$ be the $k$-th eigenvalue of \eqref{eq:saddleevp} $($thus \eqref{eq:orieig}$)$, with $M(\lambda)$ being its invariant subspace; let $(\tilde\lambda_h,(\tilde u_h,\tilde\uphi{}_h,\tilde p_h,\tilde w_h))$ be the $k$-th eigenpair of \eqref{eq:saddleevpdis} generated by the Algorithm \ref{alg:variational}. Provided the assumptions in Lemma \ref{lem:deltahhh}, then, for $\mathcal{T}_{h_0}$ sufficiently fine,
$$
|\tilde\lambda_h-\lambda|\leqslant C\hat{\delta}(\widetilde{M}_N(\mu),M(\mu))\leqslant C'h^{2\tau},
$$
and there exists a $u\in H^2_0(\Omega)$ being an eigenvector of \eqref{eq:orieig} belonging to $\lambda$, such that
$$
\|\tilde u_h-u\|_{1,\Omega}\leqslant C'h^{\tau}.
$$
\end{theorem}

\begin{corollary}
Let $\widetilde{M}_N(\mu)$ be the approximation of $M(\mu)$ generated by the Algorithm \ref{alg:variational}.
\begin{enumerate}
\item In case $V_h$ is constructed by the Lagrangian type triple $P_k\sim P_k\sim P_{k-1}$, if $M(\lambda)\subset (H^{k+1}(\Omega)\times \undertilde{H}{}^{k+1}(\Omega)\times H^{k}(\Omega)\times H^{k+1}(\Omega))\cap V$, then for $\mathcal{T}_{h0}$ fine enough,
$$
\delta(M(\mu),V_h) \leqslant C'h^k.
$$
\item In case $V_h$ is constructed by the reduced Lagrangian type triple $P_2\sim P_2\sim P_0$, if $M(\lambda)\subset (H^{2}(\Omega)\times \undertilde{H}{}^{2}(\Omega)\times H^{1}(\Omega)\times H^{2}(\Omega))\cap V$, then for $\mathcal{T}_{h0}$ fine enough,
$$
\delta(M(\mu),V_h)\leqslant C'h.
$$
\end{enumerate}
\end{corollary}
Namely, an $\mathcal{O}(h^{2k})$ convergence rate can be expected on eigenvalue for the multi-level algorithm implemented with $P_k\sim P_k\sim P_{k-1}$ triple, and an $\mathcal{O}(h^2)$ rate for eigenvalue with $P_2\sim P_2\sim P_0$ triple. For eigenfunctions, the order can be the half of that for eigenvalues.

\begin{remark}
In every step of the multi-level algorithm, we only have to solve a source problem to the accuracy of $\delta(M(\mu),V_{h_i})$, which is enough to guarantee the final accuracy of the multi-level algorithm.
\end{remark}

\subsection{Implement issue and optimal complexity}

%\subsubsection{On solving the source problem}
The cost of the algorithm comes via two sources. To solve an eigenvalue problem on $\widetilde{V}_{h_i}$ for $N+1$ times, and to solve a source problem on $V_{h_i}$ every step.
Particularly, in each step of the multi-level algorithm, we have to solve a source problem: find $(u_h,\uphi{}_h,p_h,w_h)\in V_h$,  such that
\begin{equation}
\left\{
\begin{array}{cccccll}
&&&(\nabla w_h,\nabla v_h) & = &(f_h,v_h) & \forall\,v\in H^1_{h0}
\\
& (\nabla \uphi{}_h,\nabla \upsi{}_h) & +(p_h,\rot \upsi{}_h) & + (\nabla w_h,\upsi{}_h) &=&0 & \forall\,\upsi{}_h\in \undertilde{H}{}^1_{h0}
\\
& (\rot\uphi{}_h,q_h) & &&=&0 & \forall\,q_h\in L^2_{h0}
\\
(\nabla u_h,\nabla s_h)& + (\uphi{}_h,\nabla s_h) &&& = &0 &\forall\, s_h\in H^1_{h0}.
\end{array}
\right.
\end{equation}
The entire system can be decomposed to three subsystems and solved sequentially. Namely,
\begin{enumerate}
\item find $w_h\in H^1_{h0}$, such that $(\nabla w_h,\nabla v_h)=(f_h,v_h)$, $\forall\,v_h\in H^1_{h0}$;
\item find $(\uphi{}_h,p_h)\in\undertilde{H}{}^1_{h0}\times L^2_{h0}$, such that
$$
\left\{
\begin{array}{lll}
(\nabla \uphi{}_h,\nabla \upsi{}_h)  +(p_h,\rot \upsi{}_h)& =-(\nabla w_h,\upsi{}_h) & \forall\,\upsi{}_h\in \undertilde{H}{}^1_{h0}
\\
 (\rot\uphi{}_h,q_h) &=0 & \forall\,q_h\in L^2_{h0};
\end{array}
\right.
$$
\item find $u_h\in H^1_{h0}$, such that $(\nabla u_h,\nabla s_h)=-(\uphi{}_h,\nabla s_h)$, $\forall\,s_h\in H^1_{h0}$.
\end{enumerate}
The three subsystems can be solved approximately within the cost $\mathcal{O}(h^{-2})$ to guarantee the accuracy $\delta(M(\mu),V_{h_i})$. Meanwhile, the eigenvalue problem on $\widetilde{V}_{hi}$ can be solved with the cost $\mathcal{O}(\dim(\widetilde{V}_{hi}))^3$(by QR algorithm). Therefore, the total cost of the algorithm is
\begin{equation}
{\rm cost}\cequiv\sum_{i=0}^N h_i^{-2}+(N+1)(\dim(V_{h_0}))^3\leqslant \frac{1}{1-\kappa}h_N^{-2}+h_0^{-6}|\log h_{N}|.
\end{equation}
When we focus on the first several other than all eigenvalues, we can use algorithms rather than QR algorithm which costs less.  When $h_0\gg h_N$, the total cost can be $\mathcal{O}(h_N^{-2})$. The cost is optimal versus the intrinsic computational accuracy of the scheme for expected eigenvalues.

\section{Numerical experiments}
\label{sec:ne}
%{\color{red}
%\begin{itemize}
%\item The convergence rate of single level method.
%\item The multi-level method has the same convergence rate as the single level method.
%\item For both single and multi level methods, the computed eigenvalues are  upper/lower bounds by different triples.  same triple performs same for different eigenvalues.
%\end{itemize}
%}

In this section, we test the proposed mixed element scheme for eigenvalue problem (\ref{eq:orieig}) on the convex domain (unit square $\Omega=(0,1)\times(0,1)$, left of Figure \ref{fig:initmesh}) and the non-convex domain (L-shape domain $\Omega=[0,1]\times[0,1]\slash[0,\frac{1}{2}]\times[\frac{1}{2},1]$, right of Figure \ref{fig:initmesh}). The initial meshes with mesh size $h_0\approx 0.25$ are given in both of the figures, the finest mesh is obtained by five bisection refinements.

\begin{figure}[htbp]
\subfigure{\includegraphics[width=0.45\textwidth]{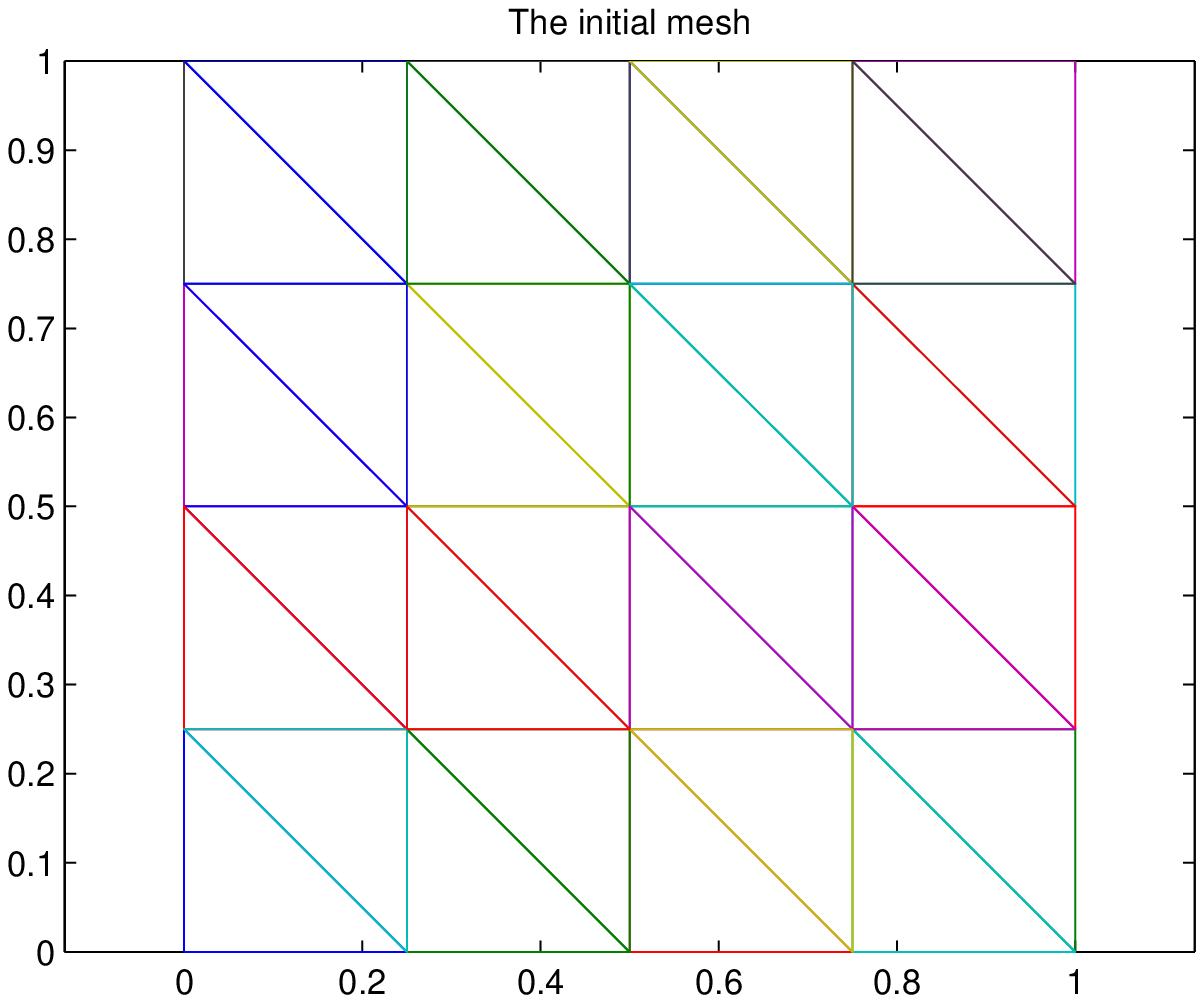}}
%\hspace{1in}
\subfigure{\includegraphics[width=0.45\textwidth]{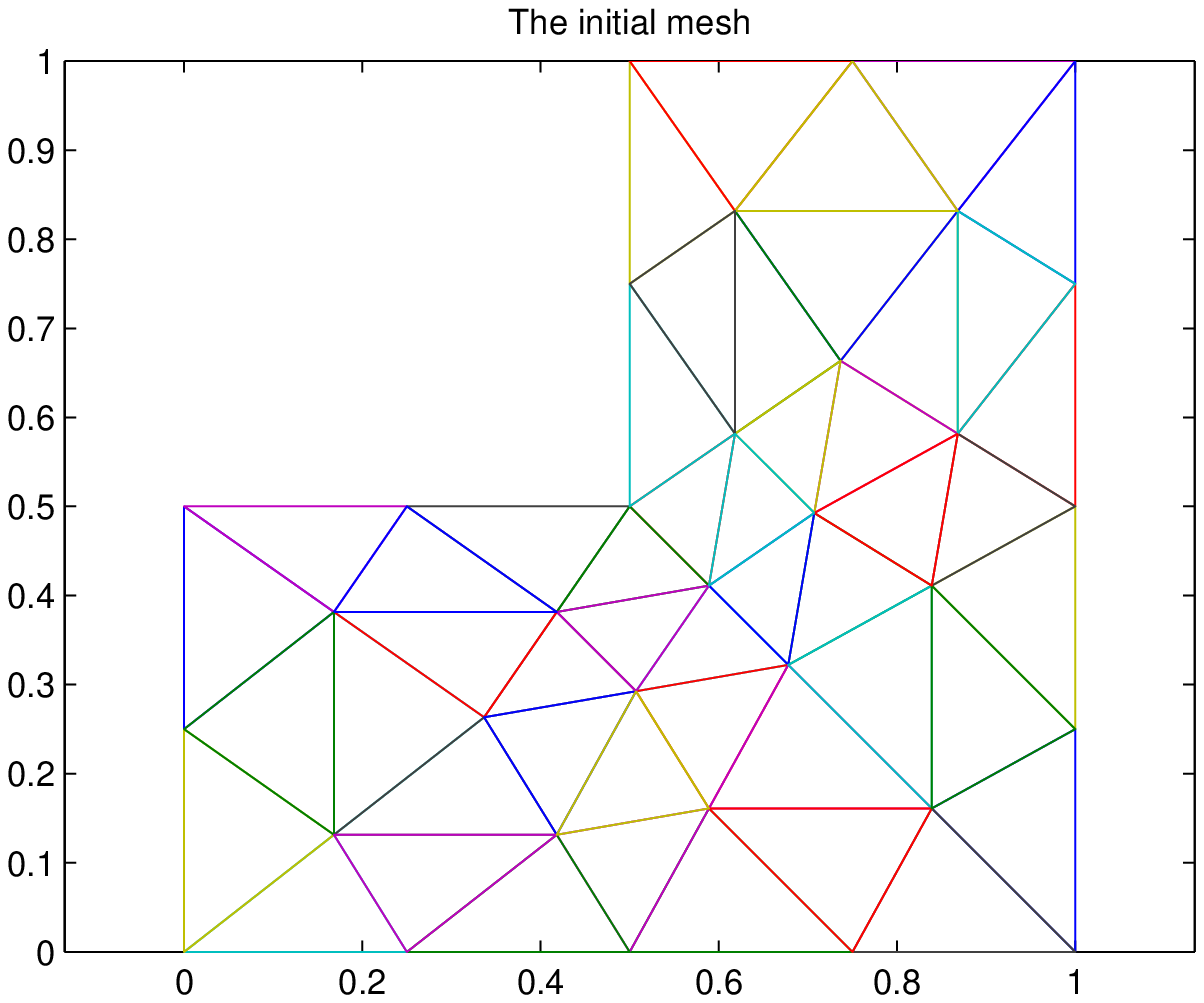}}
\caption{The initial meshes, left: the square, right: the L-shape domain.}
\label{fig:initmesh}
\end{figure}

%\begin{figure}[htbp]
%\subfigure{\includegraphics[width=0.45\textwidth]{initialmesh.eps}}
%\caption{The initial mesh of the square $\Omega=(0,1)\times(0,1)$.}
%\label{fig:initialsquare}
%\subfigure{\includegraphics[width=0.45\textwidth]{Linitmesh.eps}}
%\caption{The initial mesh of the L-shape domain $\Omega=[0,1]\times[0,1]\slash[0,\frac{1}{2}]\times[\frac{1}{2},1]$.}
%\label{fig:initialL}
%\end{figure}

We run series of numerical experiments on the these two domains, and test the accuracies of both the single-level and multi-level finite element schemes. Two kinds of finite element triples of lowest degree are tested, they are
\begin{description}
\item[triple A] the reduced Lagrangian type triples $P_2\sim P_2\sim P_0$;
%\item[{\bf triple B}] the Lagrangian type triples $P_3\sim P_3\sim P_2$;
\item[triple B] the Lagrangian type triples $P_2\sim P_2\sim P_1$.
%\item[Triple D] the nonconforming triple CR element-CR element-constant.
\end{description}

On each domain, we construct a series of nested grids $\{\mathcal{T}_{h_i}\}_{i=0}^5$ and construct finite element triples $H^1_{h_i0}\times \undertilde{H}{}^1_{h_i0}\times L^2_{h_i0}$ thereon with some specific finite elements. Particularly, we will set the grid sizes $h_i\approx h_0(1/2)^i$. On each series of meshes, we will run the single-level and multi-level algorithms, to generate two series of approximated eigenvalues $\{\lambda_{h_i}\}$ and $\{\tilde{\lambda}_{h_i}\}$, and two series of approximated eigenfunctions $\{(u_{h_i},\uphi{}_{h_i},p_{h_i},w_{hi})\}$ and $\{(\tilde{u}_{h_i},\tilde{\uphi}{}_{h_i},\tilde{p}_{h_i},\tilde{w}_{hi})\}$. The convergence order is computed by
\begin{equation}\label{ord}
Ord_{\lambda}^k=log_2(|\frac{\lambda_5-\lambda_{k-1}}{\lambda_5-\lambda_{k}}|),~~~~k=1,2,3,4,
\end{equation}
\begin{equation}\label{ord}
Ord_{u}^k=log_2(||\frac{u_5-u_{k-1}}{u_5-u_{k}}||_{H^1}),~~~~k=1,2,3,4.
\end{equation}

%\bigskip

From all these numerical results, we observe 1) both the schemes provide convergent discretization to the eigenvalue problem; their accuracy may depend on the regularity of the eigenfunctions, and essentially the domain; 2) the multi-level algorithm construct the same performance as the single-level scheme, but less computation cost if both of them use the finest mesh; 3) for {\bf triple A}, the convergence rate of eigenfunction is higher than the estimation; and 4) for both single- and multi-level methods, the computed eigenvalues can provide upper or lower bounds for the eigenvalues by different triples on convex domain.

\subsection{On the accuracy of single-level finite element schemes}

\subsubsection{Experiments on convex domain}

Figure \ref{fig:singlesqA} gives the convergence rates of the eigenvalues and eigenfunctions for the square with finite element {\bf triple A}, we give the errors for the first six eigenvalues and eigenfunctions, all the rates are almost 2, here we obtain the lower bound of the eigenvalues, the errors are given by $\lambda_{h_5}-\lambda_{h_k},\ k=1,2,3,4$, the convergence rates of the eigenfunctions are better than the theoretical result, the errors are given by $||u_{h_5}-u_{h_k}||_{H^1},\ k=1,2,3,4$.

\begin{figure}
\subfigure{\includegraphics[width=0.45\textwidth]{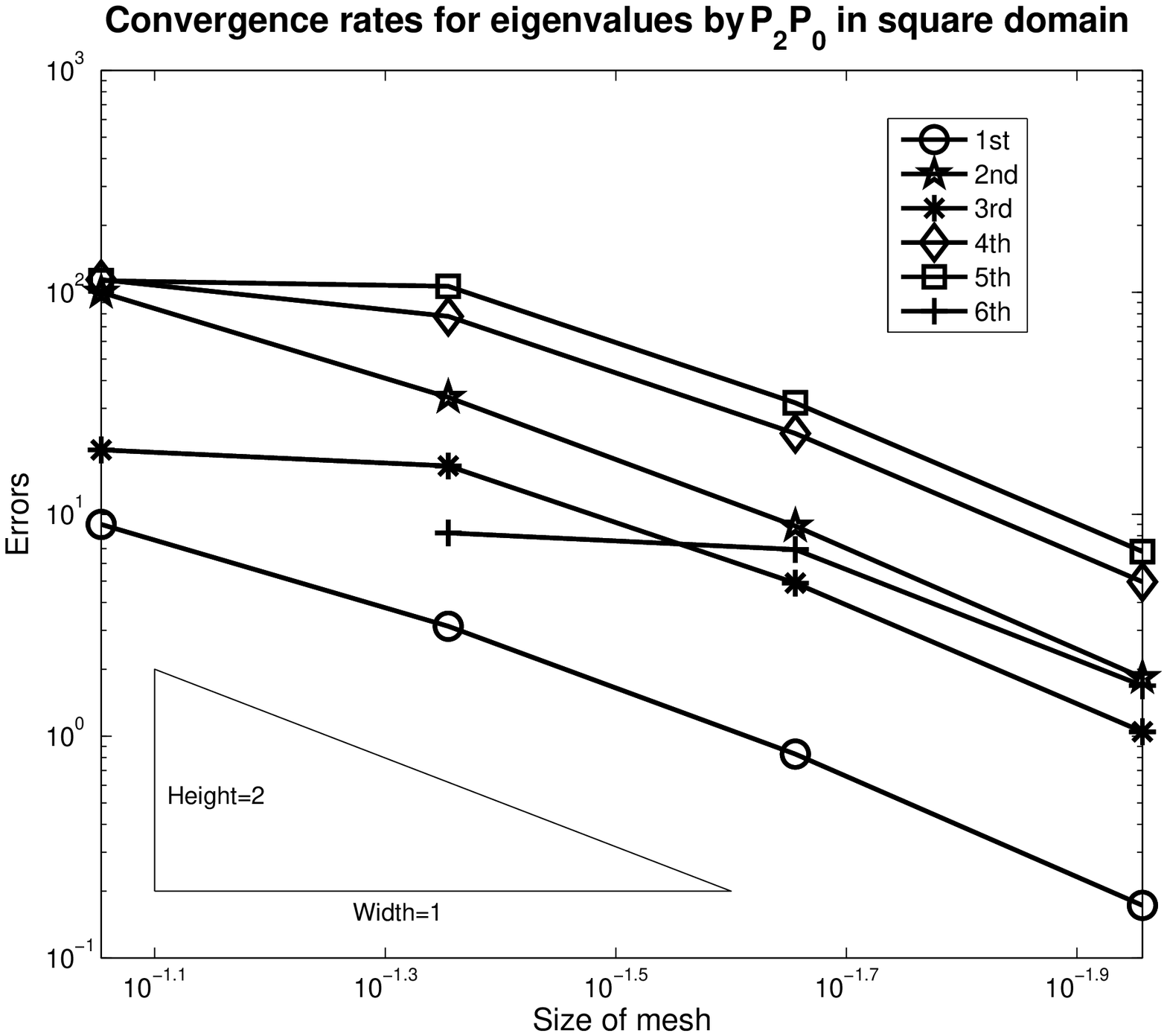}}
\subfigure{\includegraphics[width=0.45\textwidth , height=0.37\textwidth]{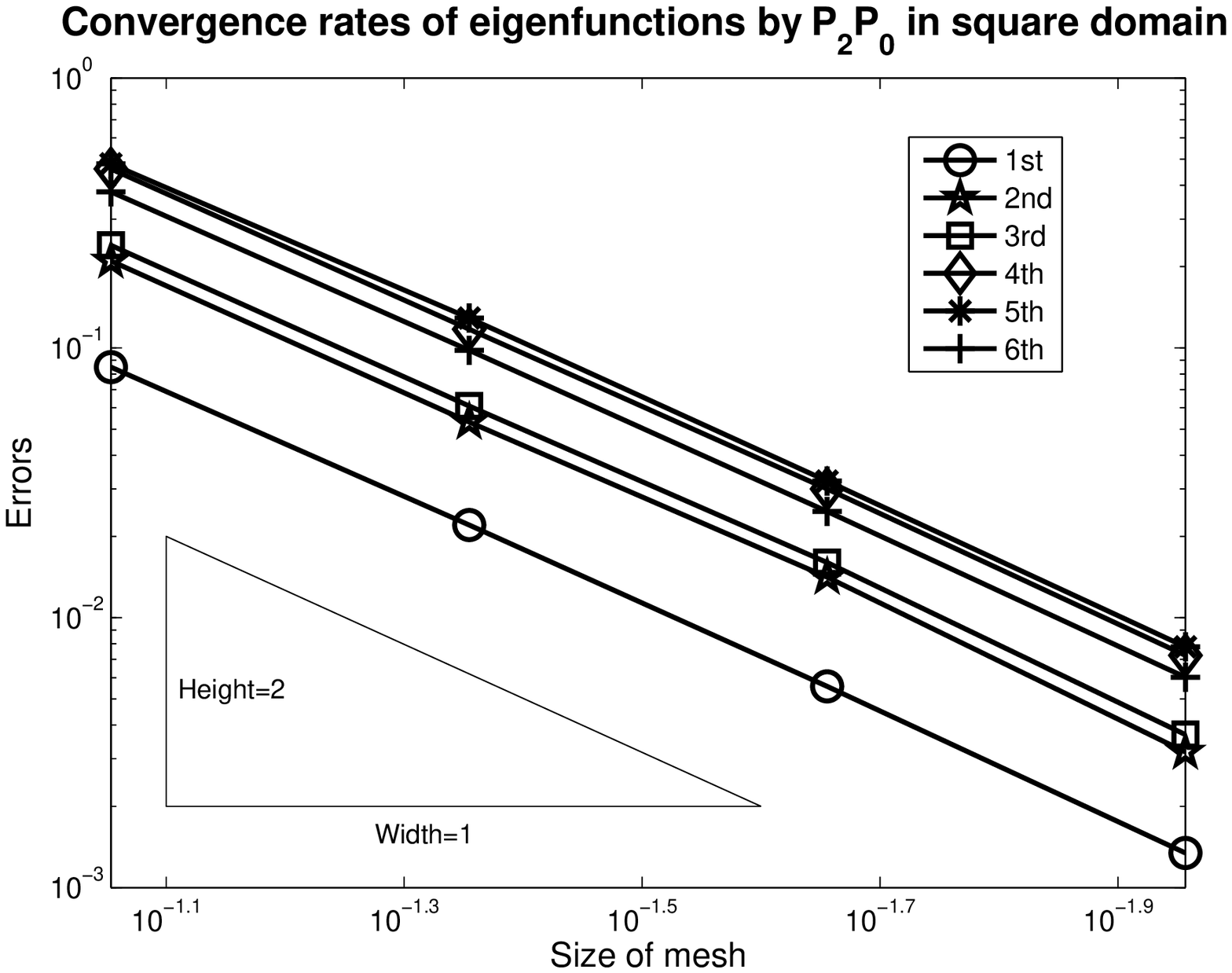}}
\caption{The convergence rates for the eigenvalues  and eigenfunctions of the square with single-level scheme and {\bf triple A}. $Y$-axis of left figure means $\lambda_{h_5}-\lambda_{h_k},\ k=1,2,3,4$, one point is missing since on the coarse mesh $\lambda_{h_5}-\lambda_{h_1}<0$.
 $Y$-axis of right figure means $||u_{h_5}-u_{h_k}||_{H^1},\ k=1,2,3,4$.
% {\color{blue} please say some more}{\color{red} y ax: log of $\lambda-\lambda_h$ ...; on finest level, $\lambda-\lambda_h<0$, this is why the evaluation is absent. ... }
}
\label{fig:singlesqA}
\end{figure}

Figure \ref{fig:singlesqB} gives the convergence rates of the the first six eigenvalues and eigenfuctions for the square with finite element {\bf triple B}, all the convergence rates of eigenvalues are almost 4, here we obtain the upper bound of the eigenvalues,
the errors are given by $\lambda_{h_k}-\lambda_{h_5},\ k=1,2,3,4$. All the convergence rates of eigefunctions are almost 2 which is consistent with the theoretical result.

\begin{figure}
\subfigure{\includegraphics[width=0.45\textwidth]{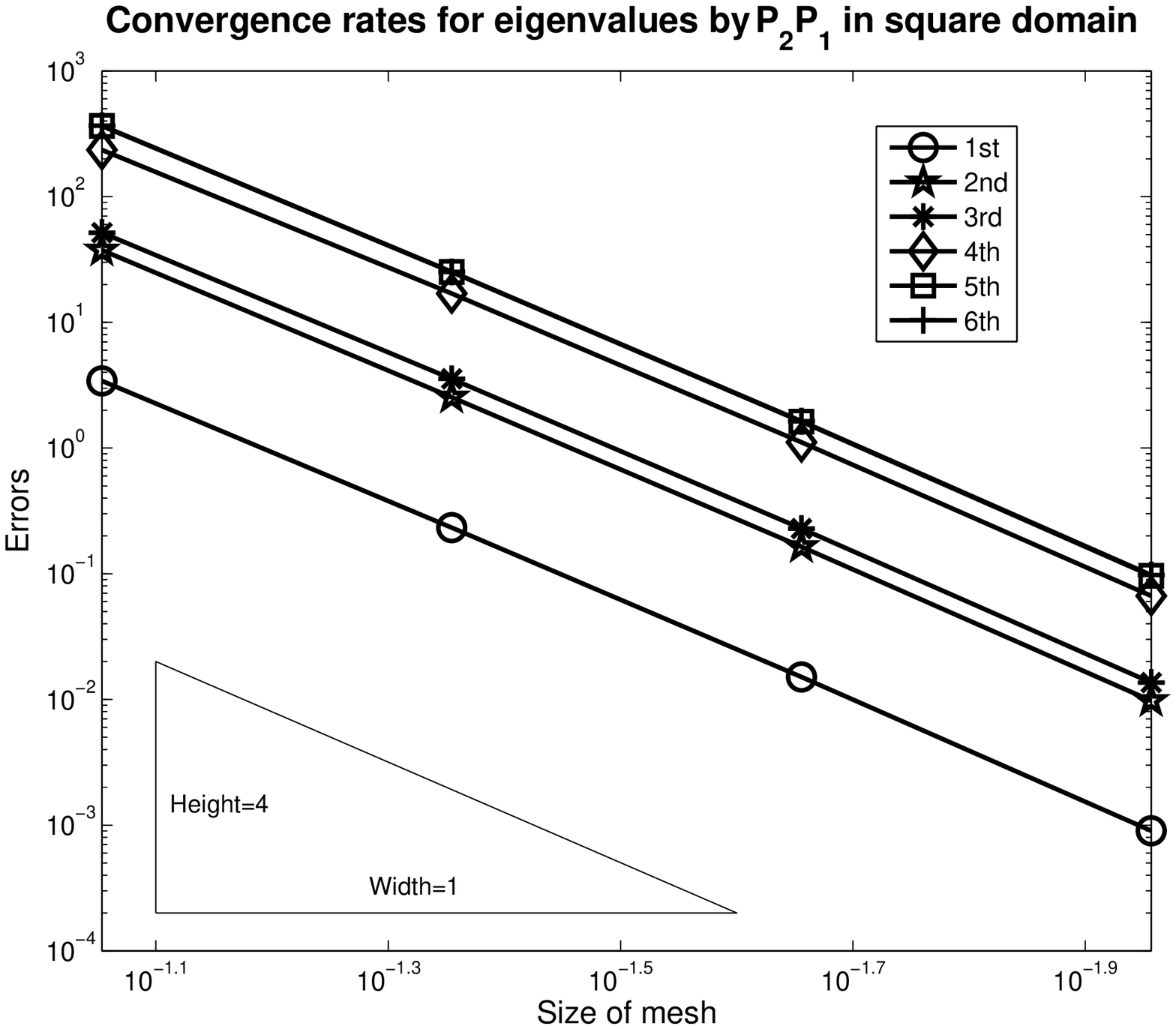}}
\subfigure{\includegraphics[width=0.45\textwidth]{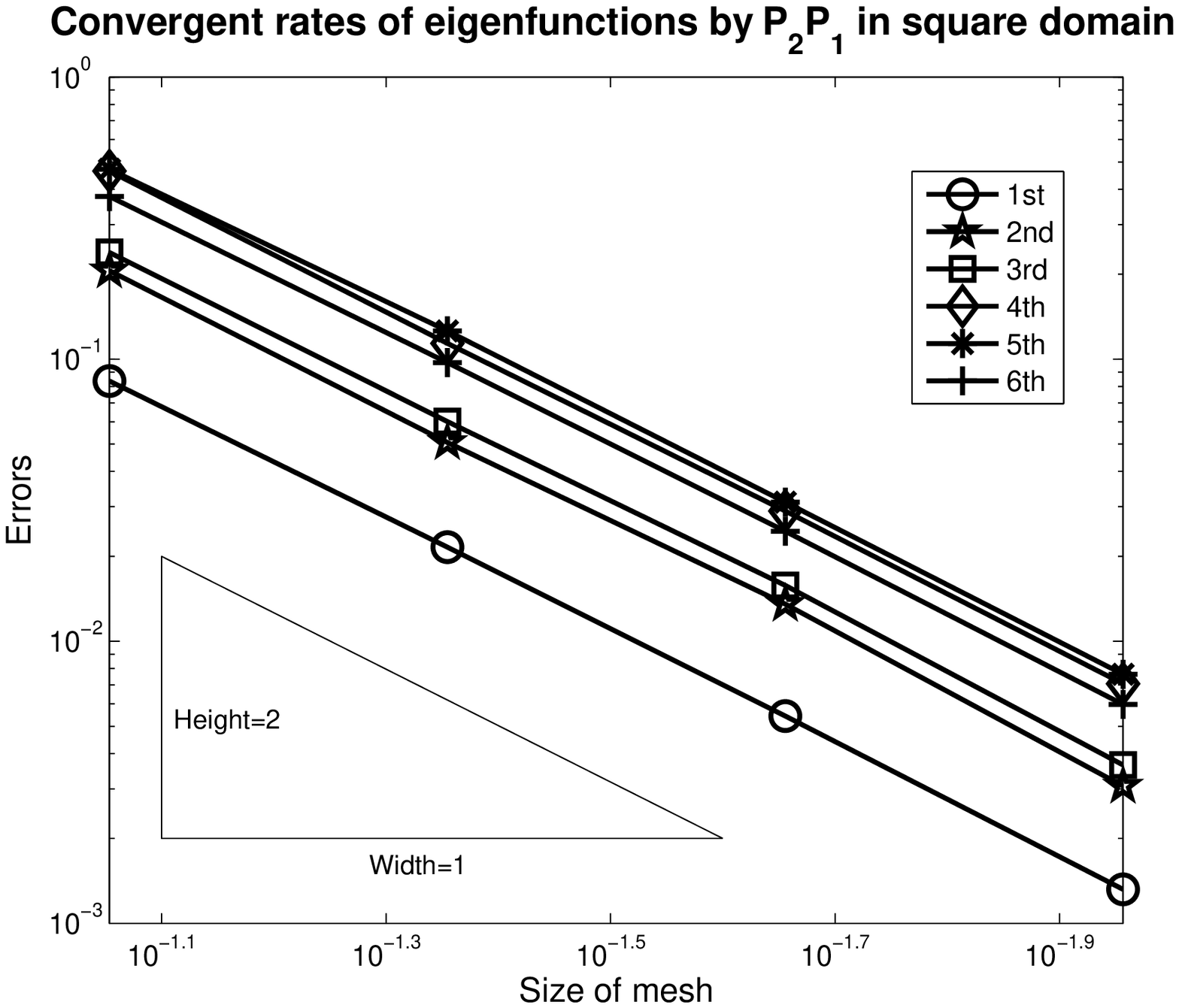}}
\caption{The convergence rates for the eigenvalues and eigenfuctions of the square with single-level scheme and {\bf triple B}. $Y$-axis of left figure means $\lambda_{h_k}-\lambda_{h_5},\ k=1,2,3,4$.  $Y$-axis of right figure means $||u_{h_5}-u_{h_k}||_{H^1},\ k=1,2,3,4$.}
\label{fig:singlesqB}
\end{figure}

\subsubsection{Experiments on nonconvex domain}

Figure \ref{fig:singlelsA} gives the convergence rates of the first six eigenvalues and eigenfuctions for the L-shape domain with finite element {\bf triple A}, all the convergence rates of the eigenvalues are almost 2, here we obtain the lower bound of the eigenvalues, the errors are given by $\lambda_{h_5}-\lambda_{h_k},\ k=1,2,3,4$. The convergence rates of the eigenfunctions are almost 2 which is better than the theoretical result.

\begin{figure}
\subfigure{\includegraphics[width=0.45\textwidth,height=0.36\textwidth]{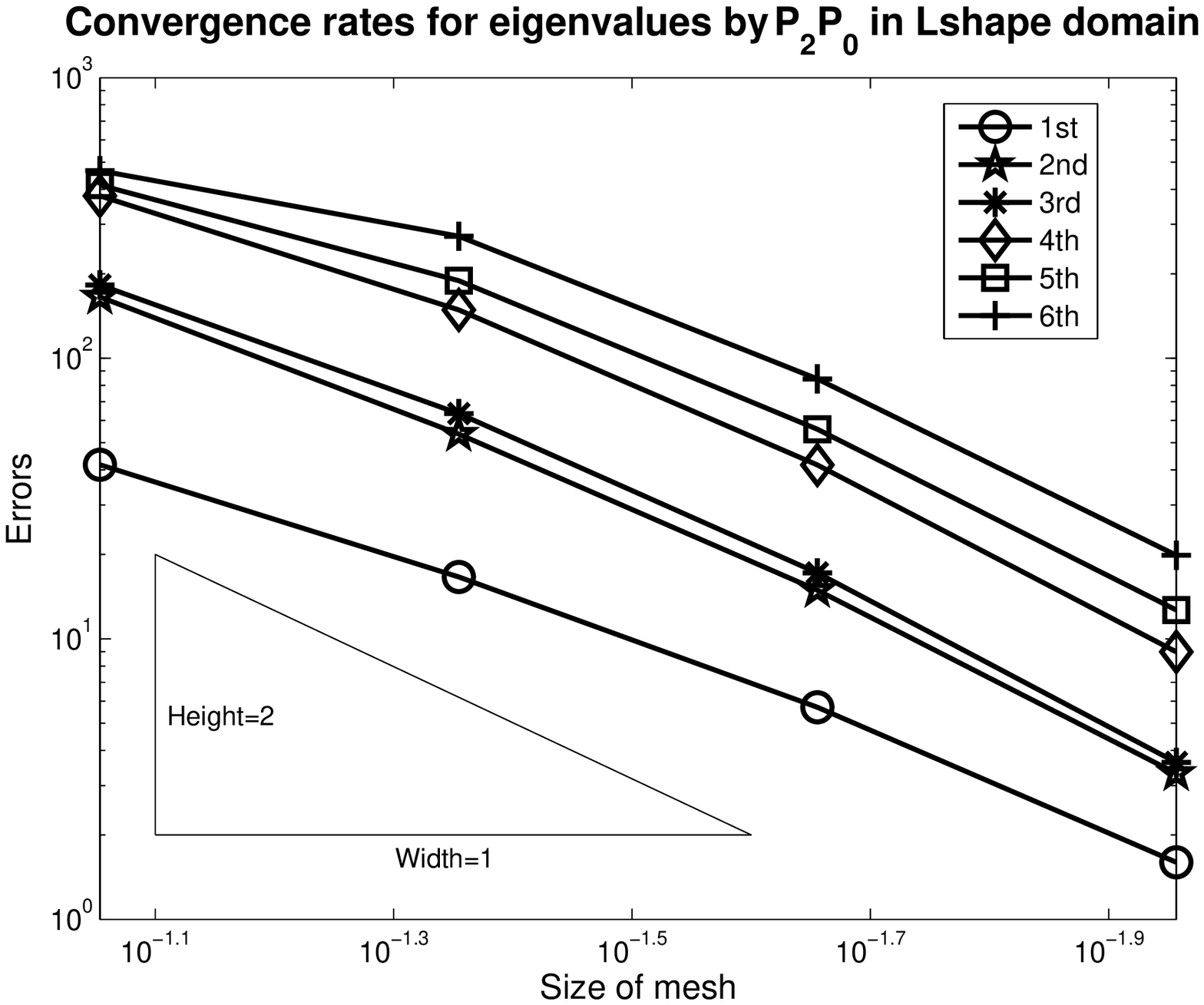}}
\subfigure{\includegraphics[width=0.45\textwidth]{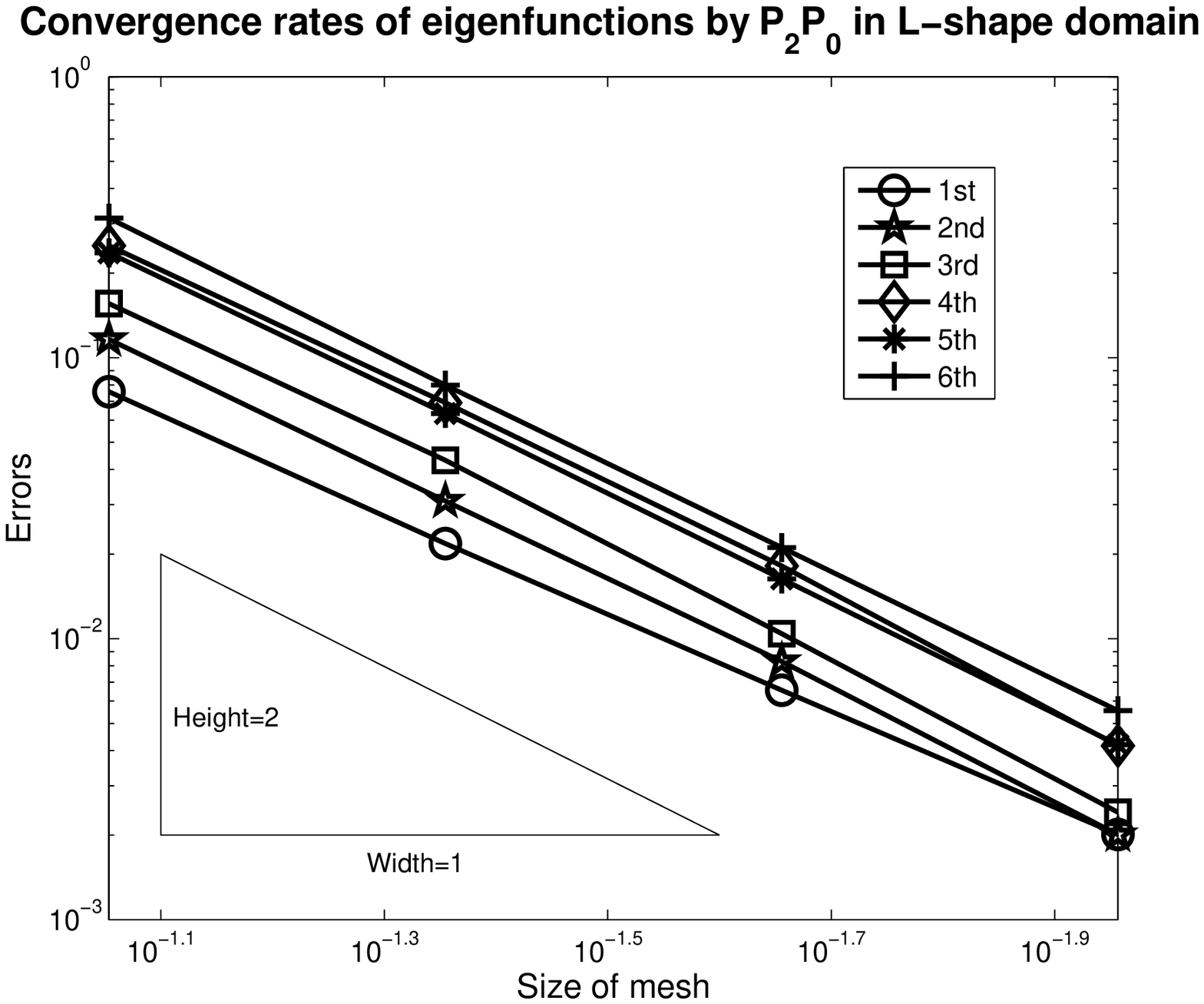}}
\caption{The convergence rates for the eigenvalues and eigenfuctions of the L-shape domain with single-level scheme and {\bf triple A}. $Y$-axis of left figure means $\lambda_{h_5}-\lambda_{h_k},\ k=1,2,3,4$.  $Y$-axis of right figure means $||u_{h_5}-u_{h_k}||_{H^1},\ k=1,2,3,4$.}
\label{fig:singlelsA}
\end{figure}

Table \ref{tab:P2P1_direct_lshape} gives the convergence rates of the the first six eigenvalues and eigenfunctions for the L-shape domain with finite element {\bf triple B}, the change of the eigenvalues is not monotone.
\begin{table}[htp]
\caption{\label{tab:P2P1_direct_lshape}The performance of {\bf triple B}  on L-shape domain with single-level scheme.}
\begin{tabular}{cccccccc}\hline
Mesh&1&2&3&4&Trend&$Ord_{\lambda}$&$Ord_u$\\\hline
$\lambda_1$&6637.38041&6671.06581&6687.93810&6696.13794&$\nearrow$&1.61242&1.64878\\%\hline
$\lambda_2$&11057.17095&11054.86661&11054.58037&11054.52410&$\searrow$&2.60578&2.06026\\%\hline
$\lambda_3$&14905.85096&14904.70082&14905.03399&14905.17967&$\searrow\nearrow$&1.71677&2.05330\\%\hline
$\lambda_4$&26165.81310&26153.57454&26152.64925&26152.55881&$\searrow$&3.48943&2.08511\\%\hline
$\lambda_5$&33343.11501&33391.54019&33423.03931&33438.85710&$\nearrow$&1.58081&1.73460\\%\hline
$\lambda_6$&53319.98768&53463.51716&53539.42249&53575.08523&$\nearrow$&1.64543&1.71939\\\hline
\end{tabular}
\end{table}

\subsection{On the accuracy of multi-level finite element schemes}
\subsubsection{Experiments on convex domain}

Figure \ref{fig:mulsqA} gives the convergence rates of the first six eigenvalues and eigenfuctions for the square with finite element {\bf triple A} by the multi-level scheme, the multi-level method has almost the same convergence rates as the single-level one, all the convergence rates are almost 2, here we also obtain the lower bound of the eigenvalues as in the single-level scheme, the errors are given by $\tilde{\lambda}_{h_5}-\tilde{\lambda}_{h_k},\ k=1,2,3,4$.

\begin{figure}
\subfigure{\includegraphics[width=0.45\textwidth]{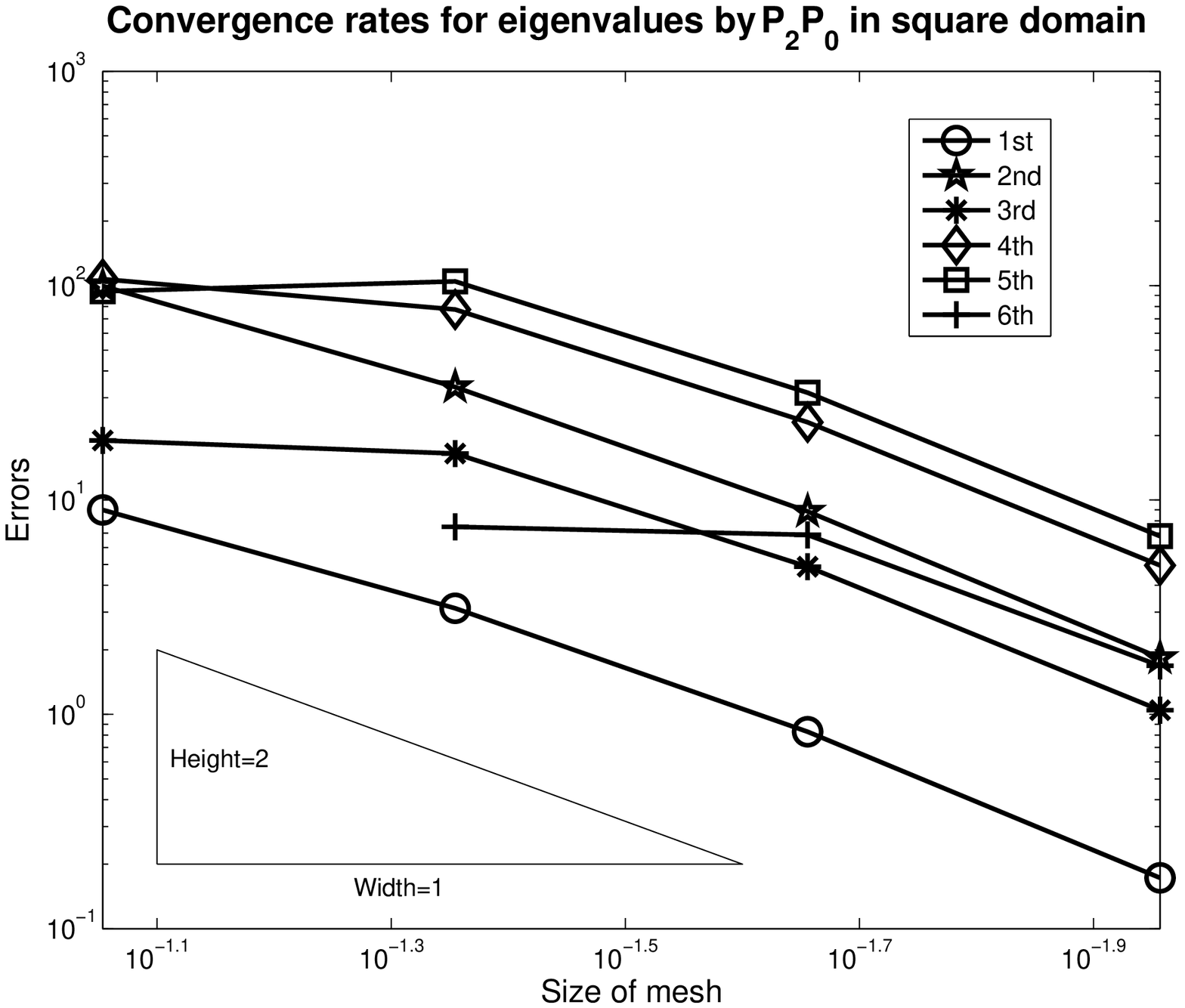}}
\subfigure{\includegraphics[width=0.45\textwidth,height=0.35\textwidth]{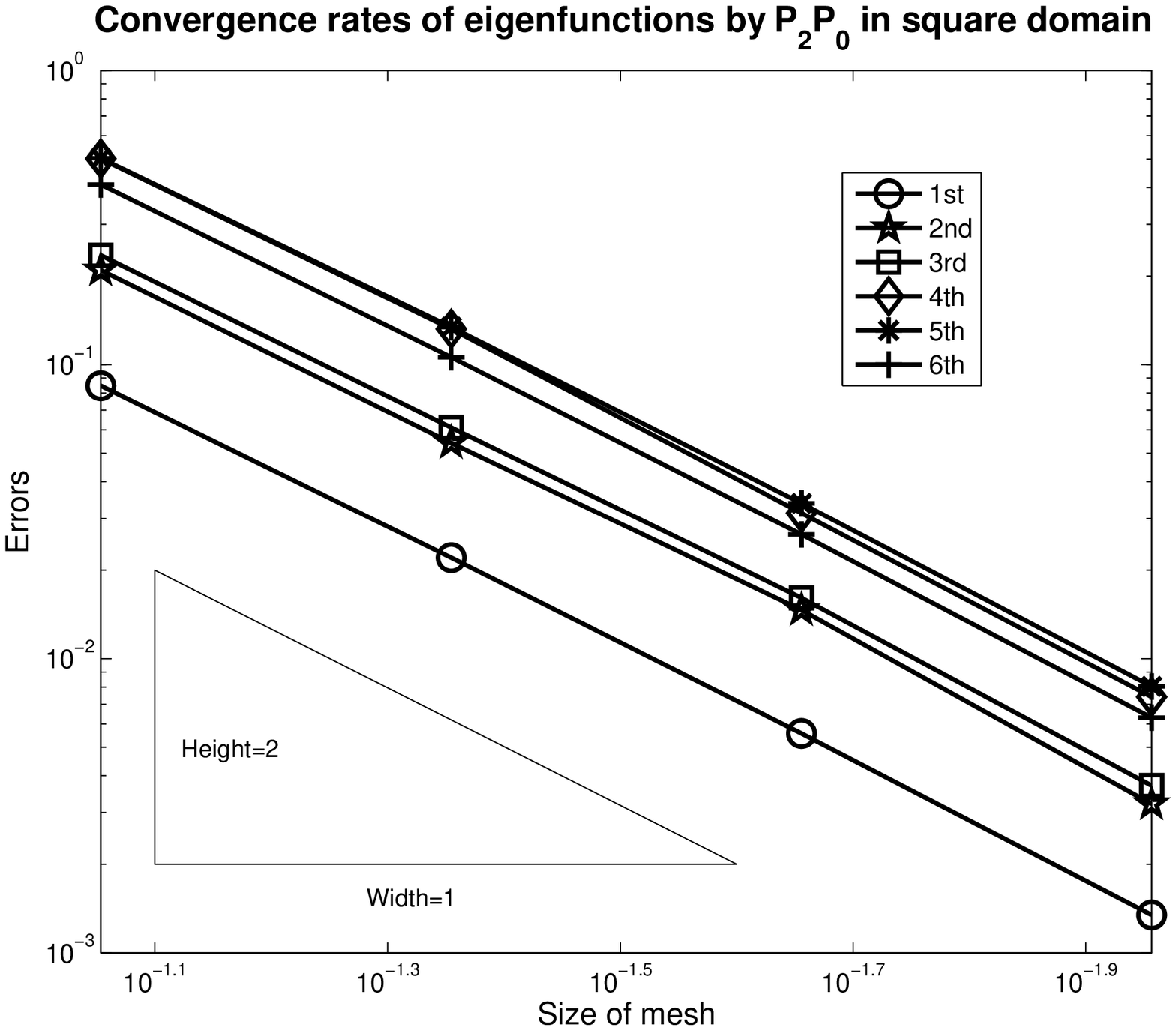}}
\caption{The convergence rates for the eigenvalues and eigenfuctions of the square with multi-level scheme and {\bf triple A}. $Y$-axis of left figure means $\tilde{\lambda}_{h_5}-\tilde{\lambda}_{h_k},\ k=1,2,3,4$, one point is missing since on the coarse mesh $\tilde{\lambda}_{h_5}-\tilde{\lambda}_{h_1}<0$.  $Y$-axis of right figure means $||\tilde u_{h_5}-\tilde u_{h_k}||_{H^1},\ k=1,2,3,4$.}
\label{fig:mulsqA}
\end{figure}

Figure \ref{fig:mulsqB} gives the results with finite element {\bf triple B}, all the convergence rates for the eigenvalues are almost 4 which is the same as single-level method and we also get the upper bound, all the convergence rates for the eigenfunctions are almost 2.

\begin{figure}
\subfigure{\includegraphics[width=0.45\textwidth]{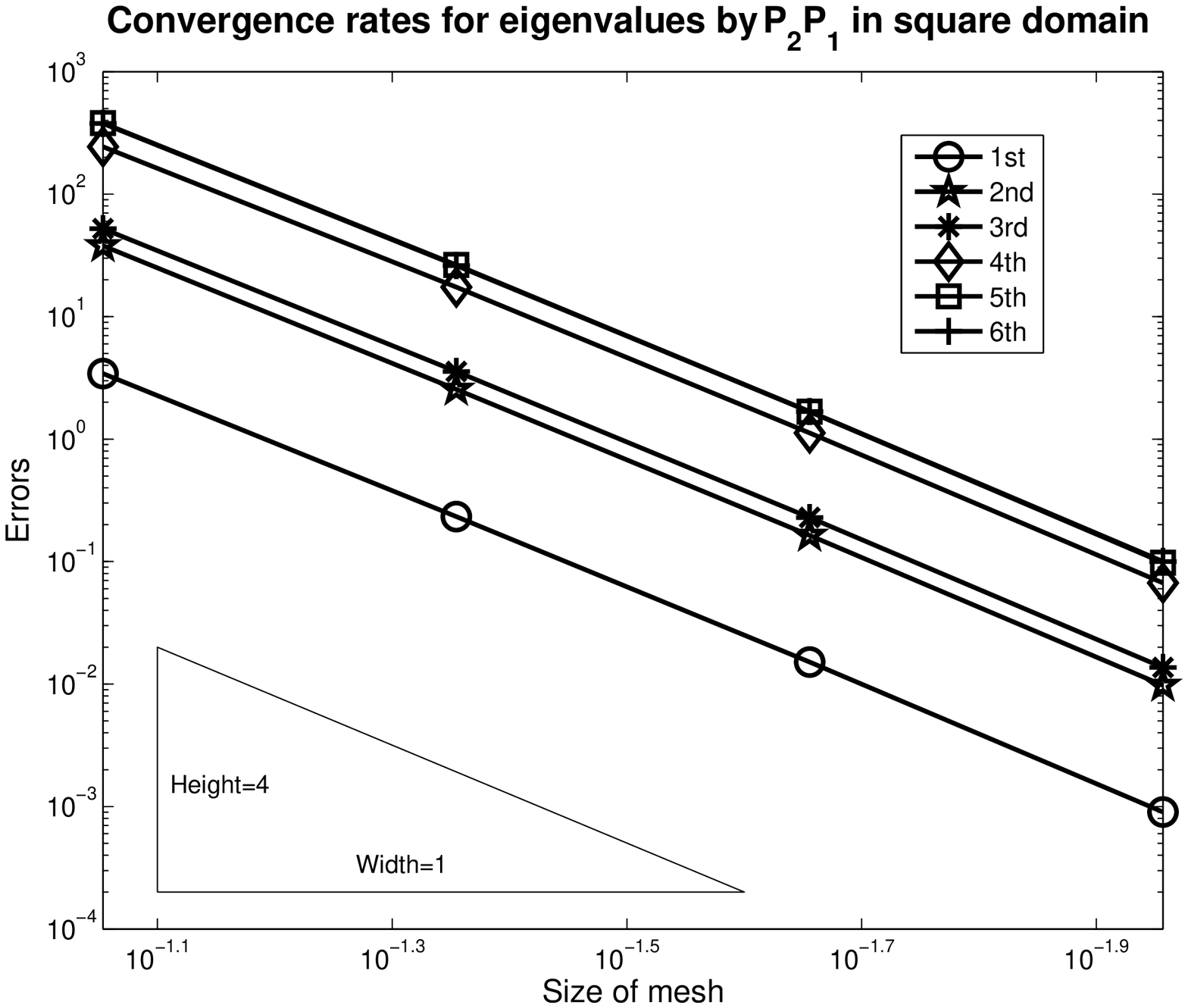}}
\subfigure{\includegraphics[width=0.45\textwidth,height=0.35\textwidth]{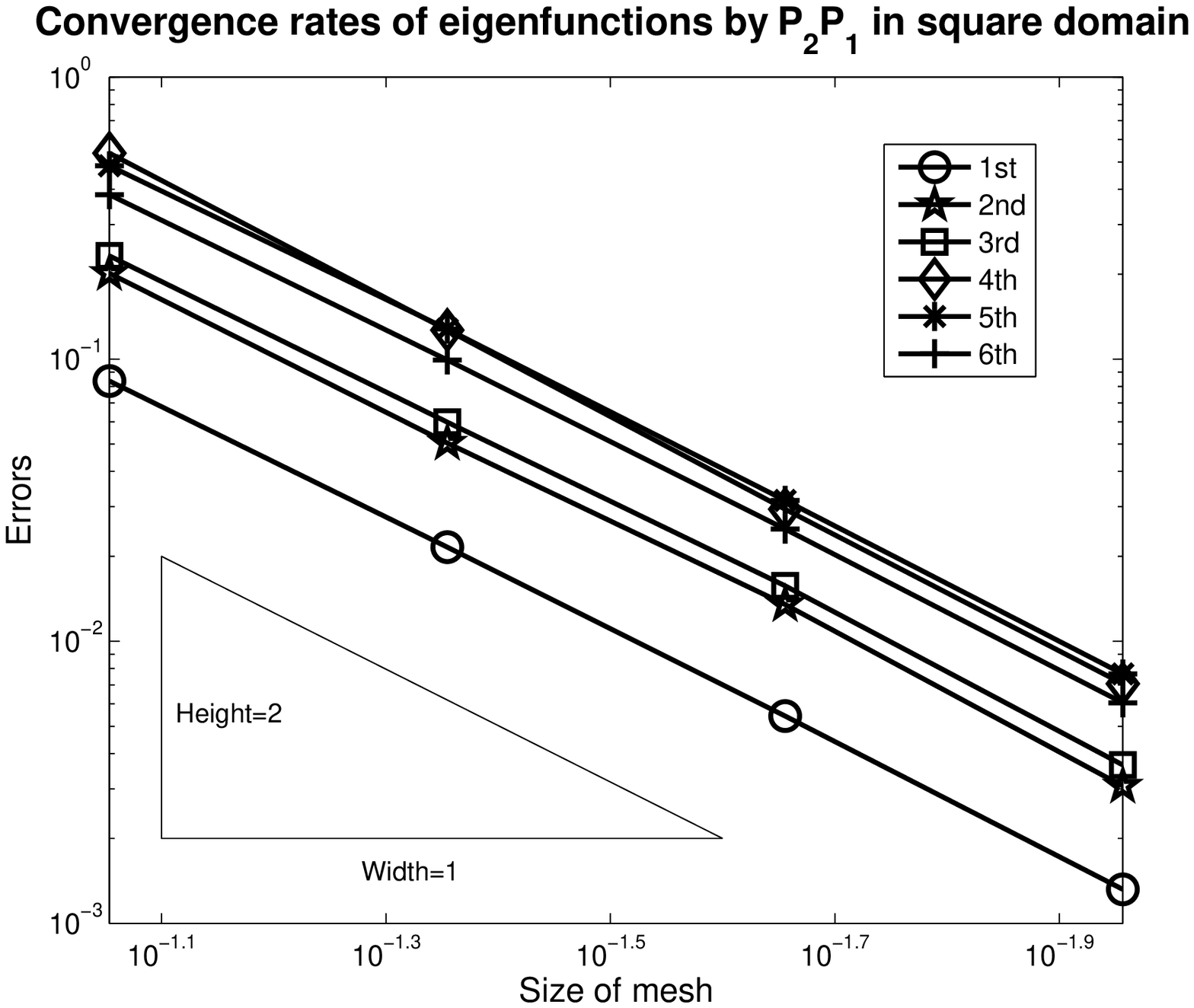}}
\caption{The convergence rates for the eigenvalues and eigenfunctions of the square with multi-level scheme and {\bf triple B}. $Y$-axis of left figure means $\tilde{\lambda}_{h_k}-\tilde{\lambda}_{h_5},\ k=1,2,3,4$. $Y$-axis of right figure means $||\tilde u_{h_5}-\tilde u_{h_k}||_{H^1},\ k=1,2,3,4$.}
\label{fig:mulsqB}
\end{figure}

\subsubsection{Experiments on nonconvex domain}

Figure \ref{fig:multilsA} gives the convergence rates of the first six eigenvalues and eigenfunctions for the L-shape domain with finite element {\bf triple A} by multi-level scheme, analogous to single-level method, all the convergence rates are almost 2 and the lower bound is obtained, which is similar to Figure \ref{fig:singlelsA}.

\begin{figure}
\subfigure{\includegraphics[width=0.45\textwidth]{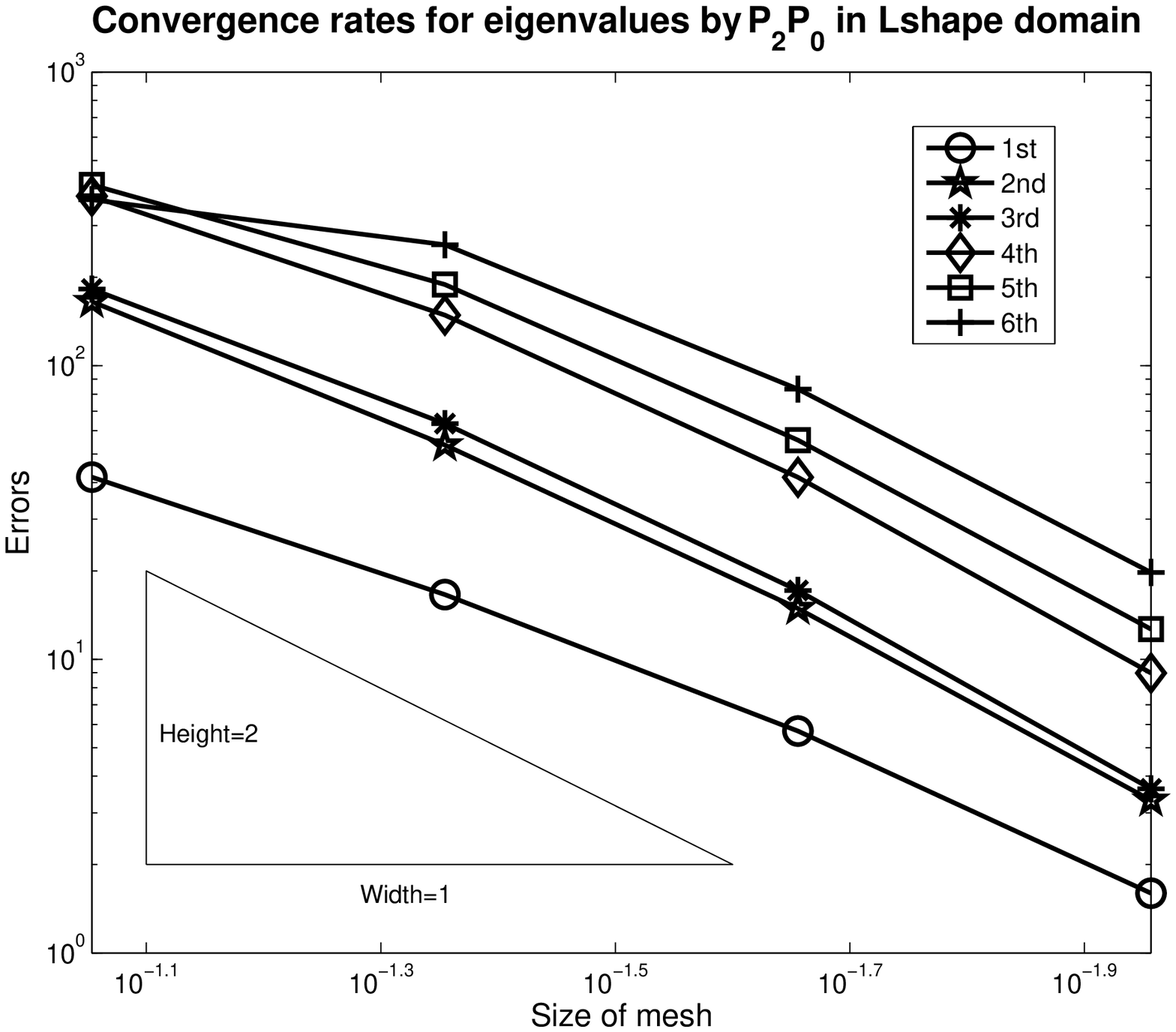}}
\subfigure{\includegraphics[width=0.45\textwidth]{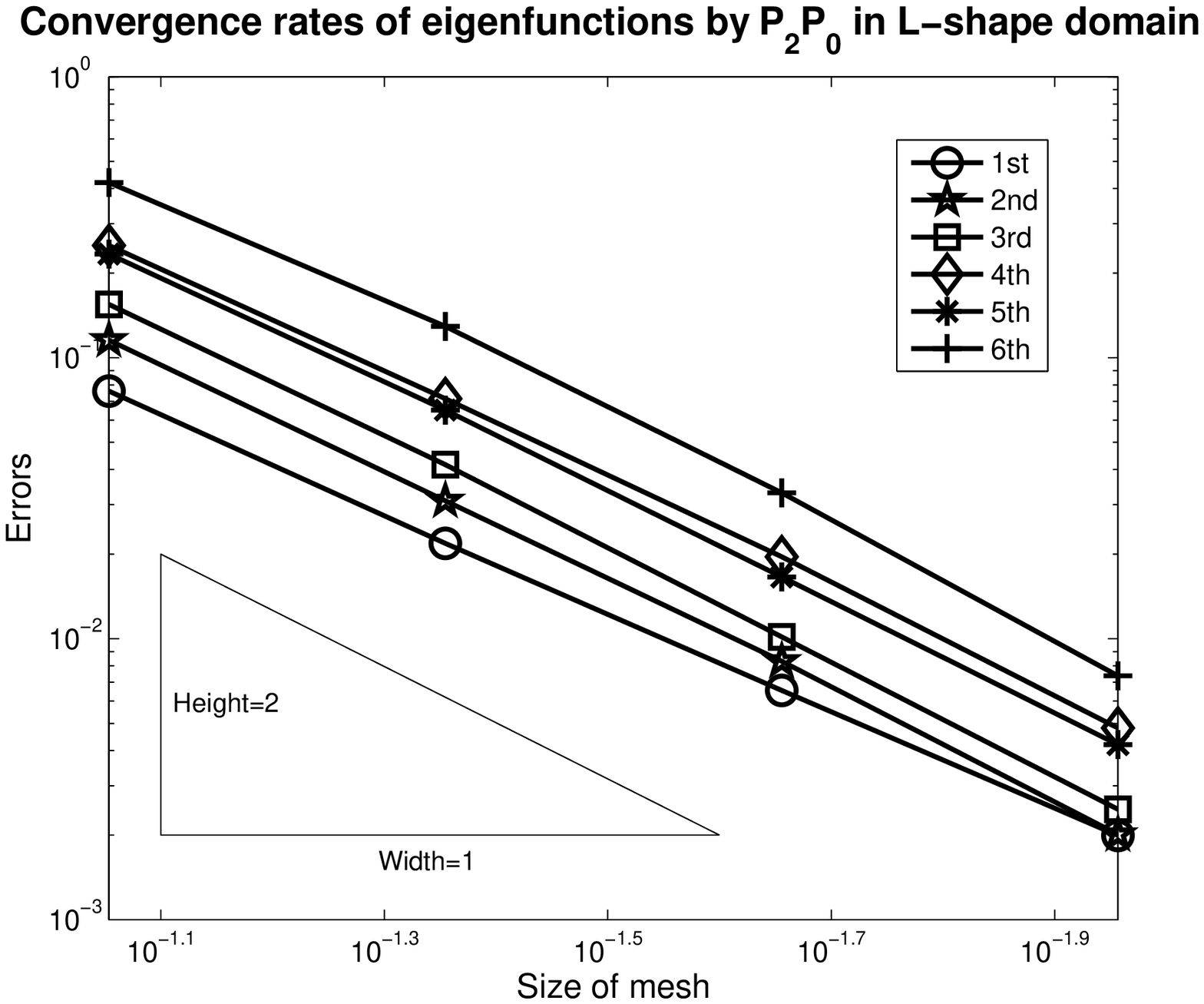}}
\caption{The convergence rates for the eigenvalues and eigenfunctions of the L-shape domain with multi-level scheme and {\bf triple A}. $Y$-axis of left figure means $\tilde{\lambda}_{h_5}-\tilde{\lambda}_{h_k},\ k=1,2,3,4$. $Y$-axis of right figure means $||\tilde u_{h_5}-\tilde u_{h_k}||_{H^1},\ k=1,2,3,4$.}
\label{fig:multilsA}
\end{figure}

Table \ref{tab:P2P1_multi_lshape} gives the convergence rates of the the first six eigenvalues and eigenfunctions for the L-shape domain with finite element {\bf triple B} by multi-level scheme, the change of the eigenvalues is still not monotone.
\begin{table}[htp]
\caption{\label{tab:P2P1_multi_lshape}The performance of {\bf triple B}  on L-shape domain with multi-level scheme.}
\begin{tabular}{cccccccc}\hline
Mesh&1&2&3&4&Trend&$Ord_{\lambda}$&$Ord_u$\\\hline
$\lambda_1$&6637.38138&6671.06594&6687.93813&6696.13795&$\nearrow$&1.61241&1.66165\\%\hline
$\lambda_2$&11057.17116&11054.86661&11054.58037&11054.52410&$\searrow$&2.60579&2.06026\\%\hline
$\lambda_3$&14905.85342&14904.70090&14905.03400&14905.17968&$\searrow\nearrow$&1.71659&1.92185\\%\hline
$\lambda_4$&26165.83290&26153.57474&26152.64926&26152.55882&$\searrow$&3.48970&2.08559\\%\hline
$\lambda_5$&33343.30473&33391.55758&33423.04243&33438.85781&$\nearrow$&1.58052&1.66333\\%\hline
$\lambda_6$&53330.17977&53465.12739&53539.64109&53575.12545&$\nearrow$&1.63222&1.68321\\\hline
\end{tabular}
\end{table}

\section{Concluding remarks}
\label{sec:cr}

In this paper, we construct a multi-level mixed scheme for the biharmonic eigenvalue problem. The algorithm possesses both optimal accuracy and optimal computational cost. We remark that, the mixed formulation given in the present paper is equivalent to the primal one; namely, at continuous level, no spurious eigenvalue is brought in. By the mixed formulation presented in this paper, the biharmonic eigenvalue problem can be discretized with low-degree Lagrangian finite elements. Discretized Poisson equation and Stokes problems also play roles in the implementation of the multi-level algorithm, which can reduce much the computational work. Both theoretical analysis and numerical verification are given.

For the theoretical analysis, we reinterpret the mixed formulation as an eigenvalue problem of a generalized symmetric operator $T$ on an augmented space $V$. This view of point may take hint to the research on other topics of these saddle-point problems; these will be discussed in future. Aiming at the multi-level algorithm, in this paper, we only discuss the conforming cases that $V_h\subset V$. The nonconforming cases that $V_h\not\subset V$ can also be used as a single-level algorithm lonely. Also, the utilization to biharmonic equation with other boundary condition and eigenvalue problems with other types can be expected.

It is observed that both the single- and multi-level algorithms tend to be able to provide upper or lower bounds of the eigenvalues, at least when the domain is convex. The theoretical verification and further utilization of this phenomena will be meaningful. Actually, the computation of the guaranteed bounds with the mixed formulation is not that trivial, as the operator associated is not adjoint in the Hilbert space. Some new techniques may have to be turned to for the theoretical analysis. Also, once we can get the guaranteed bounds, the multi-level algorithms can be improved in both its design and performance. The guaranteed computation of the upper and lower bounds will be discussed in future works. Because the mixed formulation admits nested discretization, the combination and interaction between the multi-level algorithm and the adaptive algorithm seem expected. This will also be discussed in future.
~\\

\paragraph{\bf Acknowledgement} The authors would like to thank Prof. Hehu Xie for his valuable discussion.

\end{document}